%% file: optimalcomplexityhp.tex
\let\ORIlabel\label
\let\ORIrefstepcounter\refstepcounter
\AddToHook{package/hyperref/before}{
   \let\label\ORIlabel 
   \let\refstepcounter\ORIrefstepcounter}

\documentclass[hidelinks,onefignum,onetabnum]{siamart220329}

\input{ex_shared}

\usepackage{lmodern}
\usepackage{bm}
\usepackage{microtype}
\usepackage{hyperref}
\usepackage{mathtools}
\usepackage{cleveref}
\usepackage{enumitem}
\usepackage{booktabs}
\usepackage{diagbox}
\usepackage{hhline}

\title{ Quasi-optimal complexity {$hp$}-FEM for the Poisson Equation on a rectangle }

\def\bfP{{\bf P}}
\def\bfC{{\bf C}}
\def\bfW{{\bf W}}
\def\bfx{{\bf x}}
\def\bfy{{\bf y}}
\def\bfc{{\bf c}}
\def\bfe{{\bf e}}
\def\bfu{{\bf u}}
\def\bfv{{\bf v}}
\def\bff{{\bf f}}
\def\bfH{{\bf H}}
\def\bfQ{{\bf Q}}
\def\bfzero{{\bf 0}}
\def\bfHt{\tilde\bfH}
\def\bfQuptop{\bfQ_{0:p}}
\def\bfQuptoq{\bfQ_{0:q}}
\def\bfPuptop{\bfP_{0:p}}
\def\bfPuptoq{\bfP_{0:q}}

\usepackage{todonotes}

\usepackage{todonotes}

\input{somacros}

\begin{document}

\maketitle

\begin{abstract}
We show, in one dimension, that an $hp$-Finite Element Method ($hp$-FEM) discretisation can be  solved in optimal complexity because the discretisation has a special sparsity structure that ensures that the \emph{reverse Cholesky factorisation}---Cholesky starting from the bottom right instead of the top left---remains sparse. Moreover, computing and inverting the factorisation  { may } parallelise across different elements. By incorporating this approach into an Alternating Direction Implicit (ADI) method à la Fortunato and Townsend (2020) we can solve, within a prescribed tolerance, an $hp$-FEM discretisation of the (screened) Poisson equation on a rectangle with quasi-optimal complexity: $O(N^2 \log N)$ operations where $N$ is the maximal total degrees of freedom in each dimension. When combined with fast Legendre transforms we can also solve
nonlinear time-evolution partial differential equations in a quasi-optimal complexity of $O(N^2 \log^2 N)$ operations, which we demonstrate on the (viscid) Burgers' equation. { We also demonstrate how the solver can be used as an effective preconditioner for PDEs with variable coefficients, including coefficients that support a singularity.}
\end{abstract}

\section{Introduction}
Consider the classic problem of solving the (screened) Poisson equation in a rectangle:
\begin{align}
-\Delta u(x,y) + \omega^2 u(x,y) = f(x,y) \qfor a \leq x \leq b, \;  c \leq y \leq d
\label{eq:poisson}
\end{align}
where $\Delta u  \coloneqq u_{xx} + u_{yy}$ is the Laplacian and we assume vanishing Dirichlet or Neumann boundary conditions. An effective and fast approach to solving this equation is the Fast Poisson Solver: using finite-differences to discretise the partial differential equation (PDE), we can diagonalise the discretisation using the Discrete Cosine Transform in a way that leads to quasi-optimal complexity, that is, $O(N^2 \log N)$ operations where $N$ is the maximal degrees of freedom along each dimension.

	In this paper we introduce an alternative approach that also achieves quasi-optimal complexity but for a high order ($hp$) framework. We utilise the work of Babu{\v{s}}ka and Suri \cite{babuvska1991efficient}, which introduced a basis for the Finite Element Method (FEM) built from tensor products of piecewise integrated Legendre polynomials  that achieved sparse discretisations  for constant  coefficient PDEs on rectangles. A fact that, perhaps, has  been inadequately emphasised is that  this approach enables quasi-optimal\footnote{There are false misconceptions in the literature originating in \cite{orszag1979spectral} that optimal complexity is $O(p^{d+1})$ operations as $p \rightarrow \infty$. We contend that optimal is $O(p^d)$ operations, and indeed our quasi-optimal complexity outperforms the misreported ``optimal'' complexity.} application of the discretisation when  combined with fast Legendre transforms \cite{alpert1991fast,keiner2011fast,townsend2018fast}: the complexity is $O(p^2 n^2 \log^2 p)$ for a discretisation of a tensor product of $p$-degree polynomials where the rectangle is subdivided into $n^2$ rectangles (that is $h = 1/n$ on the unit rectangle). Applying the discretisation quasi-optimally in an iterative framework is, therefore, a solved problem.

Inverting the discretisation is another story.  While Fortunato and Townsend \cite{fortunato2020fast}  introduced the first spectral\footnote{\, ``First'' and ``spectral'' are up to debate: it is only spectral if the solution itself is smooth. But certainly ``Fast'' is an undisputed adjective.} Fast Poisson Solver, which achieves quasi-optimal complexity  for solving the 2D Poisson equation with the aforementioned basis using the Alternating Direction Implicit (ADI) method, it was only  applicable when there was a single element. The aim of this work is to extend their approach to an arbitrary number of elements in a manner that is robust to $h$- and $p$-refinement.

The ingredients that made \cite{fortunato2020fast} successful were:
\begin{enumerate}
\item A fast solve for one-dimensional discretisations.
\item Control on the separation of the spectrum from the origin. 
\end{enumerate}
For (1) we introduce an optimal complexity $hp$-FEM solver in 1D in \secref{Cholesky} for Symmetric Positive Definite (SPD) problems: the complexity is $O(p n)$ where $p$ is the polynomial degree and $n$ is the number of elements.  For (2) we observe that the smallest eigenvalue can be computed in optimal complexity and we prove bounds built on known $hp$-FEM results that guarantee that it has the needed behaviour to achieve quasi-optimal complexity. 

\bigskip

Sparse $p$- and $hp$-FEM have a rich history. They can be traced to the work of Szab\'o \cite{szabo1979}, see also \cite[Ch.~2.5.2]{szabo2011introduction} and \cite[Ch.~3.1]{Schwab1998}. Extensions to two dimensions were further developed by Babu{\v{s}}ka and Suri~\cite{babuvska1991efficient} and Beuchler and Sch\"oberl~\cite{Beuchler2006}, where they construct a $p$-FEM on quadrilaterals and simplices, respectively. Other works of a similar theme include \cite{babuvska1994p, pavarino1993, Beuchler2007, Beuchler2012, Jia2022, Beuchler2012b, Dubiner1991, Karniadakis2005, Schwab1998} and \cite[App.~A]{Snowball2020}. The focus of $hp$-FEM literature is often deriving the necessary frameworks, proving optimal mesh adaptivity strategies, and obtaining exponential convergence rates \cite{Schwab1998, houston2002}.

The literature on fast solvers for the Poisson equation is  extensive.  To name but a few techniques: Fast Fourier Transforms (FFT), cyclic reductions \cite{Buzbee1970}, fast direct solvers%
\, for boundary element and multipole methods  \cite{Mckenney1995, Martinsson2019}, pseudospectral Fourier with polynomial subtraction \cite{Averbuch1998, Braverman1998}, the fast diagonalization method \cite{Lynch1964}, multigrid methods \cite{Brubeck2022, Gholami2016, Huismann2019, Lottes2005,schoberl2008}, and domain decomposition \cite{gillman2014direct, gillman2012direct, martinsson2009fast}. Almost always there is a tradeoff between asymptotic complexity, speed, and flexibility of the methods, e.g.~the structure required in the mesh. To our knowledge, except for the solver described in this work, there exists no fast Poisson solver in 2D that simultaneously (1) converges spectrally when the solution is smooth, (2) can mesh the domain into rectangular elements and, therefore, efficiently capture discontinuities in the data and (3) asymptotically requires only $O(N^2 \log N) = O( (pn)^2 \log (pn))$ operations for the solve and $O(N^2 \log^2 N) = O( (pn)^2 \log^2 (pn))$ operations for the setup.

\begin{remark}There are many other high-performance techniques to solve partial differential equations on similar meshes as we consider.
For moderate choices of $p$ these may outperform the method we propose, and so for many applications our contribution may be mostly theoretical. However, to our knowledge none of the existing work achieves optimal complexity for both $h$ and $p$ refinement, hence for very large $p$ our approach will outperform existing techniques. In particular, multigrid techniques that report $O(p^{d+1})$ flops for the solve \cite{Brubeck2022} (we achieve $O(p^2 \log p)$ in 2D) or $h$ and $p$-independent Krylov iteration counts \cite{schoberl2008} do not discuss the complexity of the setup (such as the quadrature for assembling the load vectors). Eventually the setup will become a bottleneck in such solvers unlike in our method where the setup remains quasi-optimal. A thorough investigation of the choice of $p$ where our approach becomes competitive would require a distributed memory implementation in order to  effectively parallelise the algorithm, which is beyond the scope of this current paper.
\end{remark}

\bigskip

The structure of the paper is as follows:

\secref{integratedLegendre}: We review the integrated Legendre functions of \cite{babuvska1991efficient} (see also \cite[Ch.~3.1.4]{Schwab1998}) and see how they lead to discretisations of  differential operators with a special sparsity structure which we call  {\it Banded-Block-Banded Arrowhead  ($B^3$-Arrowhead) Matrices}.

\secref{Discretisation}: We explain how the Poisson equation can be recast as a simple linear system involving a $B^3$-Arrowhead matrix in 1D and a Sylvester equation involving $B^3$-Arrowhead matrices in 2D. 

\secref{Cholesky}: We show that a reverse Cholesky factorisation---a factorisation of a matrix as $A = L^\top L$ where $L$ is lower triangular--- for $B^3$-Arrowhead matrices can be computed and the inverse applied in optimal complexity. Moreover, the factorisation { potentially } parallelises between different elements. 

\secref{ADI}: We discuss the ADI method and how it can be used to solve the (screened) Poisson equation in quasi-optimal complexity.  This requires spectral analysis of the underlying discretisations to control the number of iterations needed in ADI.

\secref{Transforms}: We discuss how to transform between coefficients and values efficiently, using visicid Burgers' equation in 2D with a discontinuous initial
condition as an example.

\secref{Preconditioning}:  Our solver can be used effectively as a preconditioner for PDEs with variable coefficients, where each iteration achieves quasi-optimal complexity. We demonstrate this on a time-independent Schr\"odinger equation with a variable coefficient with a singularity, where the number of iterations is roughly independent of $h$ and $p$.

\textbf{Code availability:} The numerical experiments found in this manuscript where conducted in Julia and and can be found at ADIPoisson.jl \cite{ADIPoisson.jl2024, adipoisson.jl-zenodo}.

\section{Integrated Legendre functions}\label{Section:integratedLegendre}

In this section we introduce the one-dimensional basis of integrated Legendre functions that underlies our discretisation of the Poisson equation.

\subsection{A basis for a single interval}
Define the weighted ultraspherical/Jacobi polynomials\footnote{This definition is also equal to the ultraspherical polynomial $C_{k+2}^{(-1/2)}(x)$, but to avoid discussion of orthogonal polynomials with non-integrable weights we do not use this relationship.} as
\[
W_k(x) \coloneqq  {(1-x^2) C_k^{(3/2)}(x) \over (k+1) (k+2)} = {(1-x^2) P_k^{(1,1)}(x) \over 2(k+1)} 
\]
where $C_k^{(\lambda)}$ are orthogonal with respect to $(1-x^2)^{\lambda - 1/2}$ on $[-1,1]$ for $\lambda > -1/2$, $\lambda \neq 0$  with normalisation constant
\[
	C_k^{(\lambda)}(x) = {2^k (\lambda)_k \over k!} x^k + O(x^{k-1})
\]
where $(\lambda)_k = \lambda(\lambda+1)\cdots(\lambda+k-1)$ is the Pochammer symbol.  $P_k^{(a,b)}(x)$ are Jacobi polynomials orthogonal with respect to $(1-x)^a(1+x)^b$ on $[-1,1]$ with normalisation constant given in \cite[18.3]{DLMF}.   
  
  The choice of normalisation is chosen because it leads to the simple formula
\begin{align}
  W_k'(x) = -P_{k+1}(x) \label{eq:deriv}
\end{align}
    for the Legendre polynomials $P_k(x) \equiv C_k^{(1/2)}(x)$ \cite[18.9.16]{DLMF}. In other words, they are the integral of Legendre polynomials: up to a constant they are precisely the \emph{integrated Legendre functions} used by Babuška. They are also equivalent to the basis defined by Schwab \cite[Ch.~3.1]{Schwab1998} and utilised by Fortunato and Townsend \cite{fortunato2020fast}.

It is convenient to express this relationship in terms of \emph{quasi-matrices}, which can be viewed as matrices that are continuous in the first dimension, or equivalently as a row-vector whose columns are functions:
\[
	{\D \over \dx} \underbrace{[W_0,W_1,W_2,\ensuremath{\ldots}]}_\bfW = \underbrace{[P_0,P_1,P_2,\ensuremath{\ldots}]}_\bfP  \underbrace{\begin{bmatrix} 0 \\ -1 \\ & -1 \\ && -1 \\ &&&\ddots \end{bmatrix}}_{D_W^P}.
\]
If we have a single element in 1D we can use this basis as the test and trial basis in the weak formulation of a  a differential equation. Let $\langle \cdot, \cdot \rangle$ denote the $L^2(-1,1)$-inner product. Then the Gram/mass matrix associated with Legendre polynomials is
\[
	\ip<\bfP^\top, \bfP> \coloneqq \begin{bmatrix}
		\ip<P_0,P_0> & \ip<P_0,P_1>  & \cdots \\
		\ip<P_1,P_0>& \ip<P_1,P_1>  & \cdots \\
		\vdots & \vdots & \ddots 
	\end{bmatrix} = \underbrace{\begin{bmatrix} 2 \\ & 2/3 \\ && 2/5 \\ &&& \ddots \end{bmatrix}}_{M_P}
\]
whilst the  discretisation of the weak 1D Laplacian is diagonal:
\begin{align*}
-\Delta_W &\coloneqq \ip<(\bfW')^\top, \bfW'>  = \ip<(\bfP D_W^P)^\top, \bfP D_W^P> \\
 &= 
(D_W^P)^\top M_P D_W^P = \begin{bmatrix} 2/3 \\ &  2/5  \\ && 2/7 \\ &&& \ddots \end{bmatrix}
\end{align*}
which is another way to write the formula:
\[
\ensuremath{\langle}W_k',W_j'\ensuremath{\rangle} = \ensuremath{\langle}P_{k+1},P_{j+1}\ensuremath{\rangle} =  {2 \over 2k+3}  \ensuremath{\delta}_{kj}.
\]

The mass matrix can be deduced by using the lowering relationship:
\begin{align}
\underbrace{[W_0,W_1,W_2,\ensuremath{\ldots}]}_\bfW = \underbrace{[P_0,P_1,P_2,\ensuremath{\ldots}]}_\bfP \underbrace{\begin{bmatrix} \times  \\ 0 & \times \\ \times & 0 & \times \\ &\times &0 & \times \\ &&\ensuremath{\ddots} & \ensuremath{\ddots} & \ensuremath{\ddots} 
\end{bmatrix}}_{L_W}
\label{rec:lowering}
\end{align}
where the exact formul\ae\ for the entries is in \appref{Recurrences}. For now we focus on sparsity structure. Subsequently the mass matrix can be expressed as a truncation of an infinite pentadiagonal matrix:
 \begin{align*}
	M_W &\coloneqq
\ensuremath{\langle}\bfW^\ensuremath{\top}, \bfW\ensuremath{\rangle} = L_W^\ensuremath{\top}\ensuremath{\langle}\bfP^\ensuremath{\top}, \bfP\ensuremath{\rangle} L_W = L_W^\ensuremath{\top} M_P L_W \cr
&= \begin{bmatrix}
\times & 0 & \times \\
0 & \times & 0 & \times  \\
\times & 0 & \times & 0 & \times  \\
 & \times & 0 & \times & 0 & \ddots  \\
  && \times & 0 & \times & \ddots  \\
 & &  & \ddots & \ddots & \ddots
 \end{bmatrix}.
\end{align*}
The entries have simple explicit rational expressions, or alternatively one can view this as a product of banded matrices. The latter approach is slightly less efficient but we will use it for clarity in exposition. We can similarly find the matrix of the inner products that arise in testing with this basis:
\[
\ip<\bfW^\top, \bfP>  = L_W^\top  \ip<\bfP^\top, \bfP>    = L_W^\top M_P.
\]

\subsection{Multiple intervals}
\label{sec:multiple-intervals}
Partitioning an interval $[a,b]$  into $n$ subintervals $a = x_0 < x_1 < \cdots < x_n = b$
we can use an affine map
\[
	a_j(x) = {2x - x_{j-1}  - x_j  \over x_j - x_{j-1}}
\]
to the reference interval $[-1,1]$ to  construct mapped {\it bubble functions} as
\[
W_{kj}^\bfx(x) \coloneqq \begin{cases}
	W_k(a_j(x)) & x \in [x_{j-1},x_j] \\
	0 & \hbox{otherwise}
	\end{cases}
\]
on each interval. We combine these with the standard piecewise linear hat basis
\begin{align*}
h_0^\bfx(x) &\coloneqq \begin{cases}
	{x_1 - x \over x_1 - x_0} & x \in [x_0,x_1], \\
	0 & \hbox{otherwise},
	\end{cases} \;\;\;
h_n^\bfx(x) \coloneqq \begin{cases}
	{x- x_{n-1} \over x_n - x_{n-1}} & x \in [x_{n-1},x_n], \\
	0 & \hbox{otherwise},
	\end{cases}\\
h_j^\bfx(x) &\coloneqq \begin{cases}
	{x - x_{j-1}  \over x_j - x_{j-1}} & x \in [x_{j-1},x_j], \\
	{x_{j+1}-x \over x_{j+1} - x_j} & x \in [x_j,x_{j+1}], \\
	0 & \hbox{otherwise},
	\end{cases}	\qqfor j = 1,\ldots,n-1.
\end{align*}

The hat and bubble functions are sometimes known as the internal and external shape functions, respectively \cite[Def.~3.4]{Schwab1998}. We form a block quasi-matrix by grouping together the hat  and bubble functions of the same degree:
\[
{\bf C}^\bfx \coloneqq [\underbrace{h_0^\bfx,\ldots,h_n^\bfx}_{\bfH^\bfx} | \underbrace{W_{01}^\bfx, \ldots,W_{0n}^\bfx}_{\bfW_0^\bfx} | \underbrace{W_{11}^\bfx, \ldots,W_{1n}^\bfx}_{\bfW_1^\bfx} | \cdots].
\]
We relate this to the piecewise Legendre basis
\[
\bfP^\bfx \coloneqq [\underbrace{P_{01}^\bfx, \ldots,P_{0n}^\bfx}_{\bfP_0^\bfx} | \underbrace{P_{11}^\bfx, \ldots,P_{1n}^\bfx}_{\bfP_1^\bfx} | \cdots]
\]
where
\[
P_{kj}^\bfx(x) \coloneqq \begin{cases}
	P_k(a_j(x)) & x \in [x_{j-1},x_j] \\
	0 & \hbox{otherwise}
	\end{cases}.
\]
In what follows we often omit the dependence on $\bfx$.

\subsubsection{The mass matrix}
Restricting to each panel, our basis is equivalent to a mapped version of the one panel basis defined above, hence we can re-expand ${\bf C}^\bfx$ in terms of $\bfP^\bfx$.  First note that the mass matrix  is diagonal, which we write in block form as:
\[
	\ip<(\bfP^\bfx)^\top, \bfP^\bfx> = \underbrace{\left[\begin{array}{c | c|c |c }
		M_{11} &   &\\
		\hline
		& M_{22}   &   \\
		\hline
		& & M_{33}    \\
		\hline 
		&&  & \ddots 
	\end{array}\right]}_{M_P^\bfx},
\]
where $\langle \cdot, \cdot \rangle$ denotes the $L^2(a,b)$-inner product. Since piecewise Legendre polynomials completely decouple we can view  this matrix as a direct sum:
\[
	M_P^\bfx = \pr({x_1 - x_0 \over 2} M_P) \oplus \cdots \oplus \pr({x_n - x_{n-1} \over 2} M_P)
\]
where the direct sum corresponds to interlacing the entries of the matrix, i.e.,
\[
	\bfe_\ell^\top M_{kj} \bfe_\ell = \bfe_k^\top \pr({x_{\ell} - x_{\ell-1} \over 2} M_P) \bfe_j.
\]

Note that given a piecewise polynomial $f$, its coefficients in the basis $\bfP^\bfx$ can be expressed as:
\[
	(M_P^\bfx)^{-1} \ip<(\bfP^\bfx)^\top, f>.
\]
We use this to determine  the (block) connection matrix
\[
	\bfC = \bfP \underbrace{\left[\begin{array}{c | c|c |c | c}
		R_{00} &R_{01}  &&\\
		\hline
		R_{10} & \bfzero  &R_{12}  &\\
		\hline
		&R_{21}  & \bfzero &R_{23} \\
		\hline 
		&& \ddots & \ddots & \ddots
	\end{array}\right]}_{R_C^\bfx}
\]
where the blocks are
\begin{align}
R_{k0} &\coloneqq M_{kk}^{-1} \ip< \bfP_k^\top, \bfH> 
= \begin{bmatrix}
\times & \times \\ & \ddots & \ddots \\&& \times & \times
\end{bmatrix}\in \bbR^{n \times (n+1)}, \label{rec:conversion1} \\
R_{kj} &\coloneqq M_{kk}^{-1} \ip< \bfP_k^\top, \bfW_{j-1}> 
= \begin{bmatrix}
\times & \\ & \ddots  \\&& \times 
\end{bmatrix} \in \bbR^{n \times n},\qfor j > 0,\label{rec:conversion2}
\end{align}
with the explicit formul\ae\ for the entries given in \appref{Recurrences}.
We thus have the mass matrix
\begin{align*}
M_C^\bfx &\coloneqq  \ip<\bfC^\top, \bfC> = R_C^\top M_P R_C \\
&= \left[\begin{array}{c c c c | c c c | c c c | c c c | c c c}
\times & \times &    &        & \times &   &         & \times &   &    &&&\\
\times & \times & \ddots &         & \times & \ddots &         & \times & \ddots &    &&&\\
  & \ddots & \ddots & \times       &   & \ddots & \times       &   & \ddots & \times  &&&\\
  &    & \times & \times      &   &    & \times       &   &    & \times  &&&\\
  \hline
  \times & \times &    &       & \times &   &     &&&    & \times &   &   \\
    & \ddots & \ddots &        &   & \ddots &    &&&     &  & \ddots &   \\
     &   & \times & \times      &   &   & \times  &&&     &   &  & \times \\
	 \hline
	 \times & \times &    &  &&&     & \times &   &     &&&    &  &   &   \\
	 & \ddots & \ddots &     &&&   &   & \ddots &    &&&     &  & \ddots &   \\
	  &   & \times & \times  &&&    &   &   & \times  &&&     &   &  &  \\
	  \hline
	    &&&     & \times &   &     &&&    & \times &   &     &  &   &   \\
	  &&&   &   & \ddots &    &&&     &  & \ddots &     &  & \ddots &   \\
	   &&&    &   &   & \times  &&&     &   &  & \times    &   &  &  \\
	   \hline 
	   &&&     &  &   &     &&&    &  &   &     &  &   &   \\
	   &&&   &   &  &    && \ddots &     &  & \ddots &     &  & \ddots &   \\
		&&&    &   &   &   &&&     &   &  &     &   &  &  \\	   
\end{array}\right].
\end{align*}
where again the entries have simple rational expressions that can be deduced from the components. The structure is important here:  every block is banded, and every block not in the first row or column is diagonal. 

\subsubsection{The weak Laplacian}

Similarly, we can express differentiation as a block-diagonal matrix:
\[
	{\D \over \dx} {\bf C}  = \bfP \underbrace{\left[\begin{array}{c | c|c }
		D_{00}  &&\\
		\hline
		& D_{11}  &\\
		\hline
		&& \ddots 
	\end{array}\right]}_{D^\bfx},
\]
where the blocks are
\begin{align}
D_{00} &\coloneqq M_{\bfP_0}^{-1} \ip< \bfP_0^\top, \bfH'> 
= \begin{bmatrix}
\times & \times \\ & \ddots & \ddots \\&& \times & \times
\end{bmatrix}\in \bbR^{n \times (n+1)}, \label{rec:deriv1}\\
D_{kj} &\coloneqq M_{\bfP_k}^{-1} \ip< \bfP_k^\top, \bfW_{j-1}'> 
= \begin{bmatrix}
\times & \\ & \ddots  \\&& \times 
\end{bmatrix} \in \bbR^{n \times n},\qfor j > 0,\label{rec:deriv2}
\end{align}
with the explicit entries given in \appref{Recurrences}. We thus deduce that the weak Laplacian is also block diagonal with structure
\begin{align*}
-\Delta_C^\bfx &\coloneqq  \ip<(\bfC')^\top, \bfC'> = (D)^\top M_P D \\
&= \left[\begin{array}{c c c c | c c c | c c c }
\times & \times &    &        &  &   &         \\
\times & \times & \ddots &         &  & &        \\
  & \ddots & \ddots & \times       &   & &       \\
  &    & \times & \times      &   &    &        \\
  \hline
   &  &  &       & \times &   &     \\
   &  & &        &   & \ddots &    \\
    &  & &       &   &   & \times  \\
	 \hline
	  &  &    &  &&&     &  &   &    \\
	 &  &  &     &&&   &   & \ddots &    \\
	  &   &  &   &&&    &   &   &  
\end{array}\right].
\end{align*}

Again the structure is important: we have a block diagonal matrix whose blocks are all diagonal, apart from the first which is tridiagonal.

\subsection{Homogeneous Dirichlet boundary condition}

To enforce homogeneous Dirichlet boundary conditions we need to  drop the basis functions that do not vanish at the boundary, that is, the first and last hat function. Thus we use the following basis:
\begin{align*}
{\bfQ}^\bfx &\coloneqq \bfC^\bfx  \underbrace{\left[\begin{array}{c c c c | c c c | c c c }
0 & &    &        &  &   &         \\
1 &  &  &         &  & &        \\
  & \ddots  &  &       &   & &       \\
  &    & 1 &       &   &    &        \\
  &&0 &		&&&\\
  \hline
   &  &  &       & 1 &   &     \\
   &  & &        &   & \ddots &    \\
    &  & &       &   &   & 1  \\
	 \hline
	  &  &    &  &&&     &  &   &    \\
	 &  &  &     &&&   &   & \ddots &    \\
	  &   &  &   &&&    &   &   &  
\end{array}\right]}
_{P_{D,n}} \\
&= [\underbrace{h_1,\ldots,h_{n-1}}_{\bfHt} | \underbrace{W_{01}, \ldots,W_{n1}}_{\bfW_1} | \underbrace{W_{02}, \ldots,W_{n2}}_{\bfW_2} | \cdots].
\end{align*}
The discretised operators are the same as above but with the first and last row of the first row/column blocks removed which modifies the band structure, i.e.~we have
\begin{align*}
M_Q^\bfx &\coloneqq  \ip<(\bfQ)^\top, \bfQ> = P_D^\top  M_C P_D \\
&= \left[\begin{array}{c c c c | c c c | c c c | c c c | c c c}
\times & \times &    &        & \times &   &         & \times &   &    &&&\\
\times & \times & \ddots &         & \times & \ddots &         & \times & \ddots &    &&&\\
  & \ddots & \ddots & \times       &   & \ddots & \times       &   & \ddots & \times  &&&\\
  &    & \times & \times      &   &    & \times       &   &    & \times  &&&\\
  \hline
  \times & \times &    &       & \times &   &     &&&    & \times &   &   \\
    & \ddots & \ddots &        &   & \ddots &    &&&     &  & \ddots &   \\
     &   & \times & \times      &   &   & \times  &&&     &   &  & \times \\
	 \hline
	 \times & \times &    &  &&&     & \times &   &     &&&    &  &   &   \\
	 & \ddots & \ddots &     &&&   &   & \ddots &    &&&     &  & \ddots &   \\
	  &   & \times & \times  &&&    &   &   & \times  &&&     &   &  &  \\
	  \hline
	    &&&     & \times &   &     &&&    & \times &   &     &  &   &   \\
	  &&&   &   & \ddots &    &&&     &  & \ddots &     &  & \ddots &   \\
	   &&&    &   &   & \times  &&&     &   &  & \times    &   &  &  \\
	   \hline 
	   &&&     &  &   &     &&&    &  &   &     &  &   &   \\
	   &&&   &   &  &    && \ddots &     &  & \ddots &     &  & \ddots &   \\
		&&&    &   &   &   &&&     &   &  &     &   &  &  \\	   
\end{array}\right].
\end{align*}

As before every block is banded, and every block not in the first row or column is diagonal.  The primary difference with the $\bfC^\bfx$ case above is the bandwidths of some of the blocks.

\section{Discretisations of the screened Poisson equation}\label{Section:Discretisation}

In order to discuss the FEM discretisation of \cref{eq:poisson} we  first recast it in variational form. Let $\Omega = [a,b] \times [c,d]$ and $H^s(\Omega) \coloneqq W^{s,2}(\Omega)$ where $W^{s,q}(\Omega)$, $s > 0$, $q \geq 1$, denote the standard Sobolev spaces \cite{Adams2003}. We use $L^q(\Omega)$, $q \geq 1$, to denote the Lebesgue spaces and the notation $H^1_0(\Omega)$ for the space $\{ v \in H^1(\Omega) : v|_{\partial\Omega} = 0\}$ where $|_{\partial \Omega} : H^1(\Omega) \to H^{1/2}( \partial \Omega)$ denotes the usual trace operator \cite{Gagliardo1957}. Let $H^{-1}(\Omega) = H^1_0(\Omega)^*$ denote the dual space of $H^1_0(\Omega)$. Moreover, given a Banach space $X$ and Hilbert space $H$, then $\langle \cdot,  \cdot \rangle_{X^*,X}$ denotes the duality pairing between a function in $X$ and a functional in the dual space
$X^*$, and $\langle \cdot,  \cdot \rangle_{H}$ denotes the inner product in $H$. As in the previous section, we drop the subscript when utilising the $L^2(\Omega)$-inner product. Given an $f \in H^{-1}(\Omega)$, we may rewrite the Poisson equation in variational (weak) form as: find $u \in H^1_0(\Omega)$ that satisfies
\begin{align}
\langle \nabla v, \nabla u \rangle + \omega^2 \langle  v, u \rangle = \langle v, f \rangle_{H^1_0(\Omega), H^{-1}(\Omega)} \;\; \text{for all} \;\; v \in H^1_0(\Omega),
\label{eq:weak-form}
\end{align}
where $\nabla u = (u_x, u_y)^\top$ is the gradient operator. To construct a discretisation in the FEM framework one chooses subspaces of $H^1(\Omega)$ as {\it trial space}  (discretisation of $u$) and {\it  test space} (discretisation of $v$), specified by their bases, which are termed {\it trial} and {\it test bases}, respectively.
 
 \subsection{Screened Poisson in 1D}
 
For zero Dirichlet problems we use as both the test and trial basis  $\bfQ$ up to degree $p$:
\[
	\bfQuptop \coloneqq \bfQ \underbrace{\begin{bmatrix} I_{n-1}\\ & I_n \\ && \ddots \\ &&& I_n \\ &&& 0_{n \times n} \\ &&& \vdots \end{bmatrix}}_{I_{0:p} \in \bbR^{\infty \times N}}  = [\bfHt | \bfW_0 | \cdots | \bfW_{p-2}],
\]
where $N = p  n-1$ is the total degrees of freedom in our basis up to degree $p$. That is, we approximate the solution, for $\bfu_p \in \bbR^N$:
$
u(x) \approx u_p(x) \coloneqq \bfQuptop(x) \bfu_p
$ and represent our test functions as $v_p = \bfQuptop \bfv_p$. The discretisation of our weak formulation then becomes:
\meeq{
\ip<v_p',u_p'> + \omega \ip<v_p,u_p> =  \bfv_p^\top \ip<((\bfQuptop)')^\top,(\bfQuptop)'> \bfu_p  + \omega \bfv_p^\top \ip<(\bfQuptop)^\top,(\bfQuptop)'> \bfu_p \ccr
	 = \bfv_p^\top (- \underbrace{I_{0:p}^\top\Delta_QI_{0:p}}_{\Delta_{Q,p}} +  \omega^2  \underbrace{I_{0:p}^\top M_QI_{0:p}}_{M_{Q,p}})  \bfu_p.
}
If we further assume that we have made a piecewise polynomial approximation of our right-hand side as $f \approx f_p \coloneqq \bfPuptop \bfu_p$, computed via a fast Legendre transform (\cite{townsend2018fast} or otherwise),  the right-hand side becomes:
\meeq{
	\ip<v_p,f> =\bfv_p^\top I_{0:p}^\top \ip<\bfQ^\top, \bfP> I_{0:p} \bff_p =  \bfv_p^\top I_{0:p}^\top R_Q^\top \ip<\bfP^\top, \bfP> I_{0:p} \bff_p \ccr
	= \bfv_p^\top \underbrace{I_{0:p}^\top R_Q^\top  I_{0:p}}_{R_{Q,p}^\top}  \underbrace{I_{0:p}^\top M_P  I_{0:p}}_{M_{P,p}}  \bff_p.
}
Enforcing this equation for all $\bfv_p \in \bbR^N$ leads us to an $N \times N$  system of equations:
\[
	(-\Delta_{Q,p} +  \omega^2 M_{Q,p})  \bfu_p=  R_{Q,p}^\top M_{P,p} \bff_p.
\]

\subsection{Screened Poisson in 2D}
\label{sec:helmholtz-2d}

For 2D problems we consider partitions $a = x_1 < \cdots < x_m = b$ and $c = y_1 < \cdots < y_n = d$ with truncation up to degrees $p$ and $q$, respectively. Then our basis for zero Dirichlet problems is given by the tensor product of $\bfQ^\bfx$ and $\bfQ^\bfy$, more specifically we have
\[
	u(x,y) \approx u_{pq}(x,y) \coloneqq \bfQuptop^\bfx(x) U_{pq} \bfQuptoq^\bfy(y)^\top
\]
where we represent the unknown coefficients as a matrix $U_{pq} \in \bbR^{M \times N}$ for $M = (p+1) m-1$ and $N = (q+1) n-1$. Similarly we may express the right-hand side as
\[
	f(x,y) \approx f_{pq}(x,y) = \bfPuptop^\bfx(x) F_{pq} \bfPuptoq^\bfy(y)^\top.
\]
Consider an arbitrary test function $v_{pq}(x,y) = \bfQuptop^\bfx(x) V_{pq} \bfQuptoq^\bfy(y)^\top$. By substituting in the expressions for $u$, $v$ and $f$ into \cref{eq:weak-form}, then akin to the 1D case, we arrive at
\begin{align*}
(-\Delta_{Q,p}^\bfx) U_{pq} M_{Q,q}^\bfy &+ M_{Q,p}^\bfx U_{pq}   (-\Delta_{Q,q}^\bfy) + \omega^2 M_{Q,p}^\bfx U_{pq}  M_{Q,q}^\bfy  \notag\\
&= (R_{Q,p}^\bfx)^\top M_{P,p}^\bfx F_{pq} M_{P,q}^\bfy R_{Q,q}^\bfy.
\end{align*}
We can modify this into a Sylvester's equation:
\begin{align}
(-\Delta_{Q,p}^\bfx + {\omega^2 \over 2} M_{Q,p}^\bfx) U_{pq} M_{Q,q}^\bfy &+ M_{Q,p}^\bfx U_{pq}   (-\Delta_{Q,q}^\bfy + {\omega^2 \over 2} M_{Q,q}^\bfy)\notag\\
&= (R_{Q,p}^\bfx)^\top M_{P,p}^\bfx F_{pq} M_{P,q}^\bfy R_{Q,q}^\bfy.
\label{eq:2d:poisson}
\end{align}
In \cref{Section:ADI} we will discuss how to solve \cref{eq:2d:poisson} for $U_{pq}$ in quasi-optimal complexity.

\subsection{General boundary conditions}
Akin to more traditional finite element methods, our method supports inhomogeneous Robin boundary conditions, $\alpha u + (\nabla u) \cdot n = g$ on $\partial \Omega$, for which zero Neumann conditions $(\nabla u) \cdot n = 0$ is a special case. This is achieved by utilizing the full basis $\bfC^\bfx$ (or its tensor product in higher dimensions), rather than $\bfQ^\bfx$, in the discretisation of the variational form
\begin{align}
\langle \nabla v, \nabla u \rangle + \alpha \langle v, u \rangle_{L^2(\partial \Omega)} = \langle v, f \rangle + \langle v, g \rangle_{L^2(\partial \Omega)}.
\end{align}
Mixed boundary conditions may be handled similarly. For instance setting $\alpha = g=0$ and including $h_0^\bfx$ in the discretisation basis, but not $h_n^\bfx$, imposes a zero Neumann boundary condition on the left and a zero Dirichlet boundary condition on the right.

\section{Optimal complexity Cholesky factorisation}\label{Section:Cholesky}

As noted, the  mass matrices  $M_C$/$M_Q$ and weak Laplacians $\Delta_C$/$\Delta_Q$ have a special sparsity structure:

\begin{definition}
A {\it Banded-Block-Banded-Arrowhead  ($B^3$-Arrowhead) Matrix} $A \in \bbR^{m+pn \times m+pn}$  with block-bandwidths $(\ell,u)$ and sub-block-bandwidth $\lambda+\mu$  has the following properties:
\begin{enumerate}
\item It is a block-banded matrix with block-bandwidths $(\ell,u)$.
\item The top-left block $A_0 \in \bbR^{m \times m}$ is banded with bandwidths $(\lambda+\mu,\lambda+\mu)$.
\item The remaining blocks in the first row  $B_k \in \bbR^{m \times n}$ have bandwidths $(\lambda,\mu)$.
\item The remaining blocks in the first column  $C_k \in \bbR^{n \times m}$ have bandwidths $(\mu,\lambda)$.
\item All other blocks $D_{k,j} \in \bbR^{n \times n}$ are diagonal.
\end{enumerate}
We represent the matrix in block form
\begin{align}
\begin{bmatrix}
A_0 & B \\
C & D
\end{bmatrix} = \left[\begin{array}{c|c|c|c|c|c|c}
A_0 & B_1 & \cdots & B_u & &  \\
\hline
C_1 & D_{1,1} & \cdots & D_{1,u}  & D_{1,1+u} &  \\
\hline
\vdots & \vdots & \ddots & \ddots & \ddots  & \ddots \\
\hline
C_l & D_{\ell,1} & \ddots & \ddots & \ddots & \ddots & D_{p-u,p}  \\
\hline
& D_{\ell+1,1} & \ddots & \ddots & \ddots & \ddots & D_{p-u+1,p} \\
\hline 
&& \ddots & \ddots & \ddots & \ddots & \vdots \\
\hline
&&& D_{p,p-\ell} & D_{p,p-\ell+1} & \cdots & D_{p,p}
\end{array}\right]
\label{eq:A}
\end{align}
where $A_0,B_1,\ldots,B_u,C_1,\ldots,C_\ell$ are all banded matrices with bandwidths $(\lambda,\mu)$ whilst $D_{k,j}$ are diagonal matrices. To store the diagonal blocks in $D$  we write
\[
	D = D_1 \oplus \cdots \oplus D_n
\]
where $D_k$ are banded matrices with bandwidths $(\ell,u)$, where  as above the direct sum corresponds to interlacing the entries of the matrix.
\end{definition}

If we apply a Cholesky factorisation $A = L L^\top$ directly we will have fill-in coming from the banded initial rows/columns. The key observation is that if we use a {\it reverse Cholesky factorisation}, that is  a factorisation of the form $A = L^\top L$ which begins in the bottom right, we avoid fill-in.

\begin{theorem}
If $A$ is a Symmetric Positive Definite (SPD) $B^3$-Arrowhead Matrix which has block-bandwidth $(\ell,\ell)$ and sub-block-bandwidth $\lambda+\mu$ then it has a reverse Cholesky factorisation
\[
	A = L^\top L
\]
where $L$ is a $B^3$-Arrowhead Matrix with block-bandwidth $(\ell,0)$ and sub-block-bandwidth $\lambda+\mu$.
\end{theorem}
\begin{proof}

We begin by writing $A$ as a block matrix:
\[
	A = \begin{bmatrix} 
		A_0 & B \\
		B^\top  & D
	\end{bmatrix},
\]
where
\[
	D = D_1 \oplus \cdots \oplus D_n.
\]
The reverse Cholesky factorisation $D = \tilde L^\top \tilde  L$ can be deduced from the reverse Cholesky factorisations of $D_j = L_j^\top L_j$. In particular we have that
\[
	\tilde  L \coloneqq L_1 \oplus \cdots \oplus L_n \qqand \tilde  L^{-1} = L_1^{-1} \oplus \cdots \oplus L_n^{-1}.
\]
Now write
\[
	A = \begin{bmatrix} 
		A_0 & B \\
		B^\top  & \tilde L^\top  \tilde  L
	\end{bmatrix} = \begin{bmatrix} I & B \tilde L^{-1}  \\ & \tilde  L^\top \end{bmatrix}  \begin{bmatrix} 
		A_0 - B \tilde  L^{-1} \tilde  L^{-\top} B^\top &  \\
		& I
	\end{bmatrix} \begin{bmatrix} I \\  \tilde L^{-\top} B^\top & \tilde L \end{bmatrix}.
\]
Write $\tilde L^{-1}$ in block form
\[
	\tilde L^{-1} = \left[\begin{array}{c | c | c}
		\tilde L_{1,1} && \\
		\hline
		\vdots & \ddots \\
		\hline 
		\tilde L_{p,1}& \cdots & \tilde L_{p,p}
	\end{array}\right]
\]
noting each block is diagonal. We see that
\[
	B \tilde L^{-1} = [B_1 \tilde L_{1,1} + \cdots + B_\ell \tilde L_{\ell,1} | B_2 \tilde L_{2,2} + \cdots + B_\ell \tilde L_{\ell,2} | \cdots | B_\ell \tilde L_{\ell,\ell} | 0 | \cdots | 0 ]
\]
where each block has bandwidth $(\lambda,\mu)$. Thus 
\[
	B \tilde  L^{-1} \tilde  L^{-\top} B^\top = (B_1 \tilde L_{1,1} + \cdots + B_\ell \tilde L_{\ell,1}) ( \tilde L_{1,1} B_1^\top + \cdots +  \tilde L_{\ell,1} B_\ell^\top)  + \cdots + B_\ell \tilde L_{\ell,\ell}^2 B_\ell^\top
\]
has bandwidths $(\lambda+\mu,\lambda+\mu)$, as multiplying banded matrices adds the bandwidths. Thus $A_0 - B L^{-1} L^{-\top} B^\top = L_0^\top L_0$ also has bandwidths $(\lambda+\mu,\lambda+\mu)$ and its reverse Cholesky factor has bandwidth $(\lambda+\mu,0)$. Thus 
\[
	L =  \begin{bmatrix} L_0 \\
		\tilde L^{-\top} B^\top & \tilde L
	\end{bmatrix}
\]
is a lower triangular $B^3$-Arrowhead matrix with the prescribed sparsity.

\end{proof}

Encoded in this proof is a simple algorithm for computing the reverse Cholesky factorisation, see \algref{ReverseCholesky}.

\begin{algorithm}[tb]
\caption{Reverse Cholesky for $B^3$-Arrowhead Matrices}\label{Algorithm:ReverseCholesky}

{\noindent\bf Input:} {Symmetric positive definite $B^3$-Arrowhead Matrix $A$ with block-bandwidths $(\ell,\ell)$ and sub-block-bandwidths $\lambda+\lambda$.}

{\noindent\bf Output:} {Lower triangular $B^3$-Arrowhead Matrix with block-bandwidths $(\ell,0)$ and sub-block-bandwidths $\lambda+\lambda$ satisfying $A = L^\top L$.}

\begin{algorithmic}[1]
\FOR{$k = 1, \ldots, n$}
\STATE Compute the banded reverse Cholesky factorisations $D_k = L_k^\top L_k$.
\ENDFOR
\FOR{$k = 1, \ldots, \ell$}
\STATE Construct banded matrices 
\[
	M_k = B_k  \tilde L_{k,k} + \cdots + B_\ell \tilde L_{\ell,k}.
\]
\ENDFOR
\STATE Form the banded matrix
\[
	\tilde A_0  = M_1 M_1^\top + \cdots + M_\ell M_\ell^\top.
\]
\STATE Compute the banded reverse Cholesky factorisation $ \tilde A_0 = L_0^\top L_0$.
\STATE Return the lower triangular $B^3$-Arrowhead matrix
\[
	\begin{bmatrix} L_0 \\
		C & L_1 \oplus \cdots \oplus L_n
	\end{bmatrix}
\]
where $C^\top = [M_1 | \cdots | M_\ell | 0 | \cdots | 0]$.
\end{algorithmic}
\end{algorithm}

\begin{corollary}
If $A \in \bbR^{N \times N}$ is an SPD $B^3$-Arrowhead Matrix then the reverse Cholesky factorisation can be computed and its inverse applied in optimal complexity ($O(N)$ operations).
\end{corollary}
\begin{proof}
The reverse Cholesky factorisations of  banded matrices can be computed in optimal complexity so lines (1--3) take $O(np)$ operations. Multiplying banded matrices by diagonal matrices and adding them is also optimal complexity hence lines (4--7) take $O(\max(n,m))$ operations. Finally line (8) is another banded reverse Cholesky which is $O(m)$ operations. Hence the total complexity  of the reverse Cholesky factorisation is
$
O(np + \max(n,m)) = O(N)
$
operations.

Once the factorisation is computed it is straightforward to solve linear systems in optimal complexity:  write
\[
	L = \begin{bmatrix} L_0 \\ 
		L_1 & \tilde L
	\end{bmatrix}
	\quad\hbox{so that} \quad
	L^{-1} = \begin{bmatrix} L_0^{-1} \\
		-\tilde L^{-1} L_1 L_0^{-1} & \tilde L^{-1}
	\end{bmatrix}.
\]
Since $L_0$ is banded and
$
\tilde L = L_1 \oplus \cdots \oplus L_n
$ 
where $L_k$ are banded, their inverses can be applied in optimal complexity.
\end{proof}

In  \figref{onedtims} we demonstrate the timing\footnote{All computations performed on an M2 MacBook Air with 4 threads unless otherwise stated.} for this algorithm for solving the one-dimensional  screened Poisson equation
\[
	(-\Delta + \omega^2 )u = f
\]
with a zero Dirichlet boundary condition which is discretised via 
\[
	(-\Delta_Q + \omega^2 M_Q)u = R^\top M_P \bf f
\]
where $\bf f$ are given Legendre coefficients. 
We choose $\omega = 1$ and $f$ is random samples as these do not impact the speed of the simulation.  The first plot shows the precomputation cost: building the discretisation and computing its Cholesky factorisation, achieving optimal complexity. The second plot shows  the solve time, which is also optimal complexity. The timings of both are roughly independent of $n$, the number of elements, demonstrating uniform computational cost.

\begin{figure}[ht!]
\centering
\includegraphics[width =0.45 \textwidth]{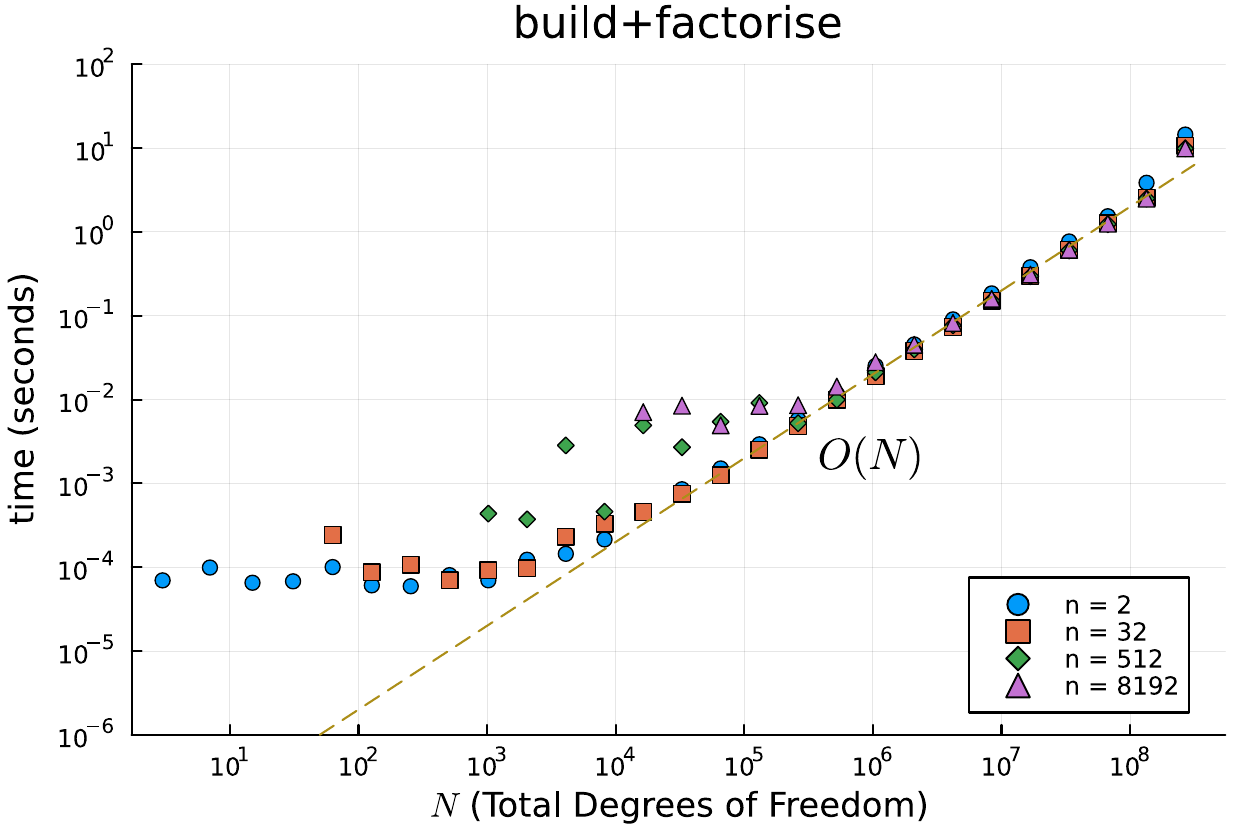}
\includegraphics[width =0.45 \textwidth]{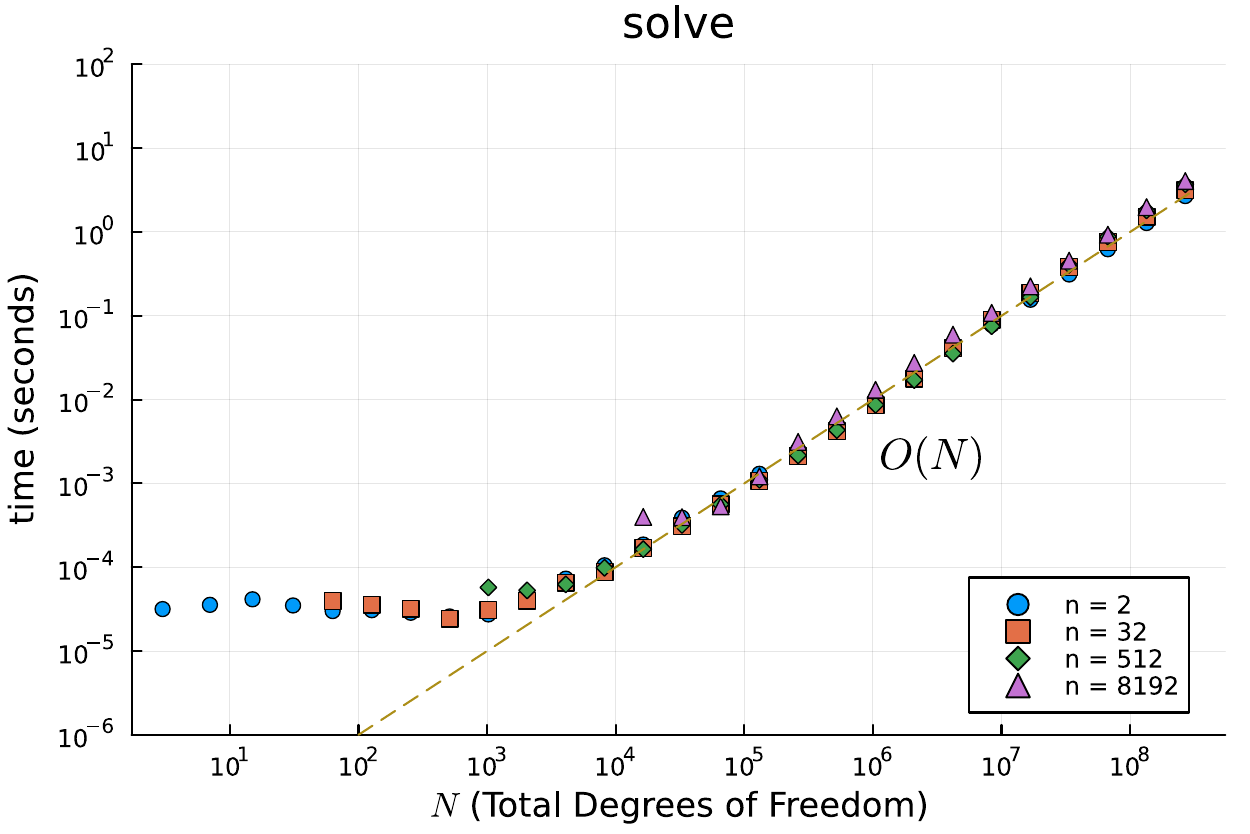}
\caption{Time taken to build/factorise and solve a discretisation of a 1D (screened) Poisson equation up to degree $p$, where $n$ is the number of elements. The $x$-axis is $N = np-1$, the total  number of degrees of freedom and demonstrates that the complexity is optimal as either $n\ (= 2/h)$ or $p$ become large, and largely only depends on the total number of degrees of freedom.  \label{Figure:onedtims}}
\end{figure}

\begin{remark}
An $O(N)$ solve for the matrix induced by the FEM discretisation of the one-dimensional screened Poisson equation is also admissible via \emph{static condensation} \cite[Ch.~3.2]{Schwab1998}.
\end{remark}

In \figref{onedtims_multithread} we demonstrate the speed-up observed in computing factorisations and solves when parallelised over multiple cores, showing evidence of strong scaling for the factorisation problem. Note in 2D the communication costs mitigate any improvement from parallelisation in the current implementation, and so to effectively take advantage of the potential to parallelise would likely require a more robust distributed memory implementation.

\begin{figure}[ht!]
\centering
\includegraphics[width =0.45 \textwidth]{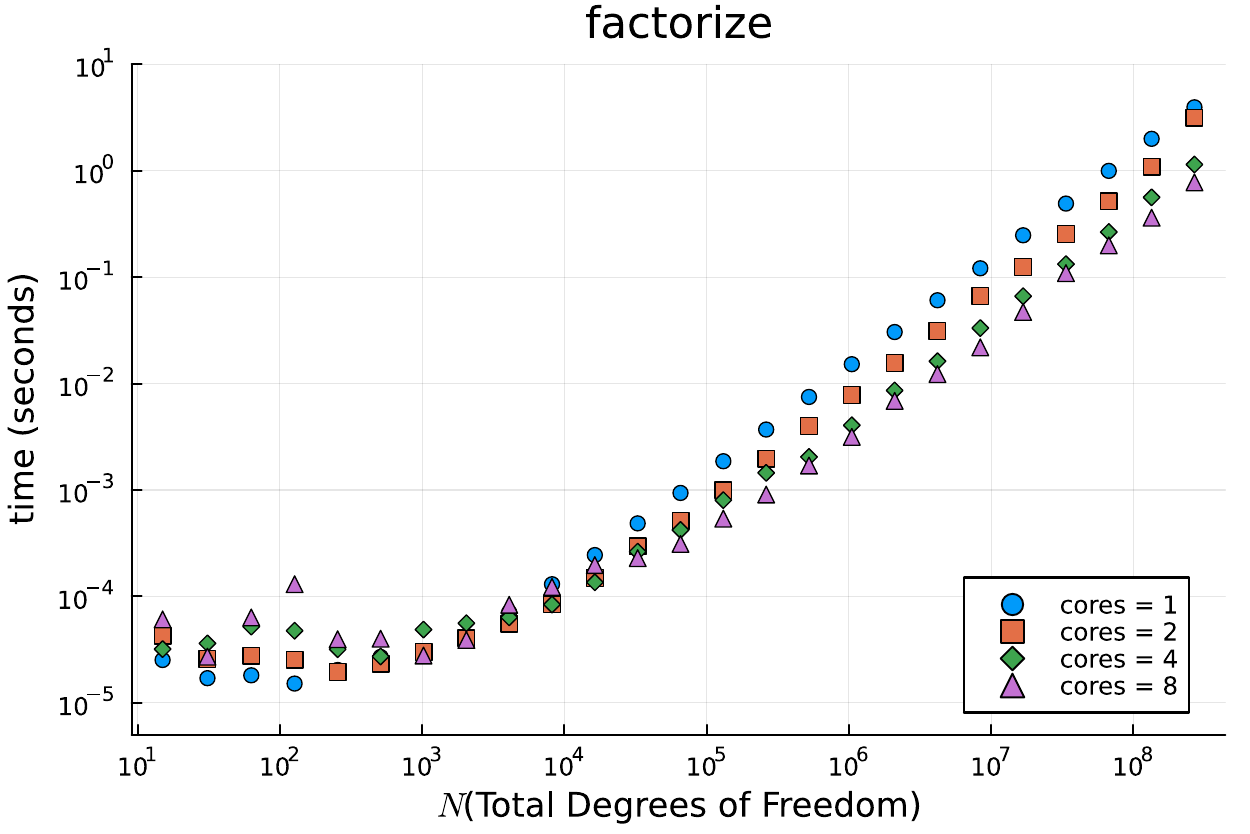}
\includegraphics[width =0.45 \textwidth]{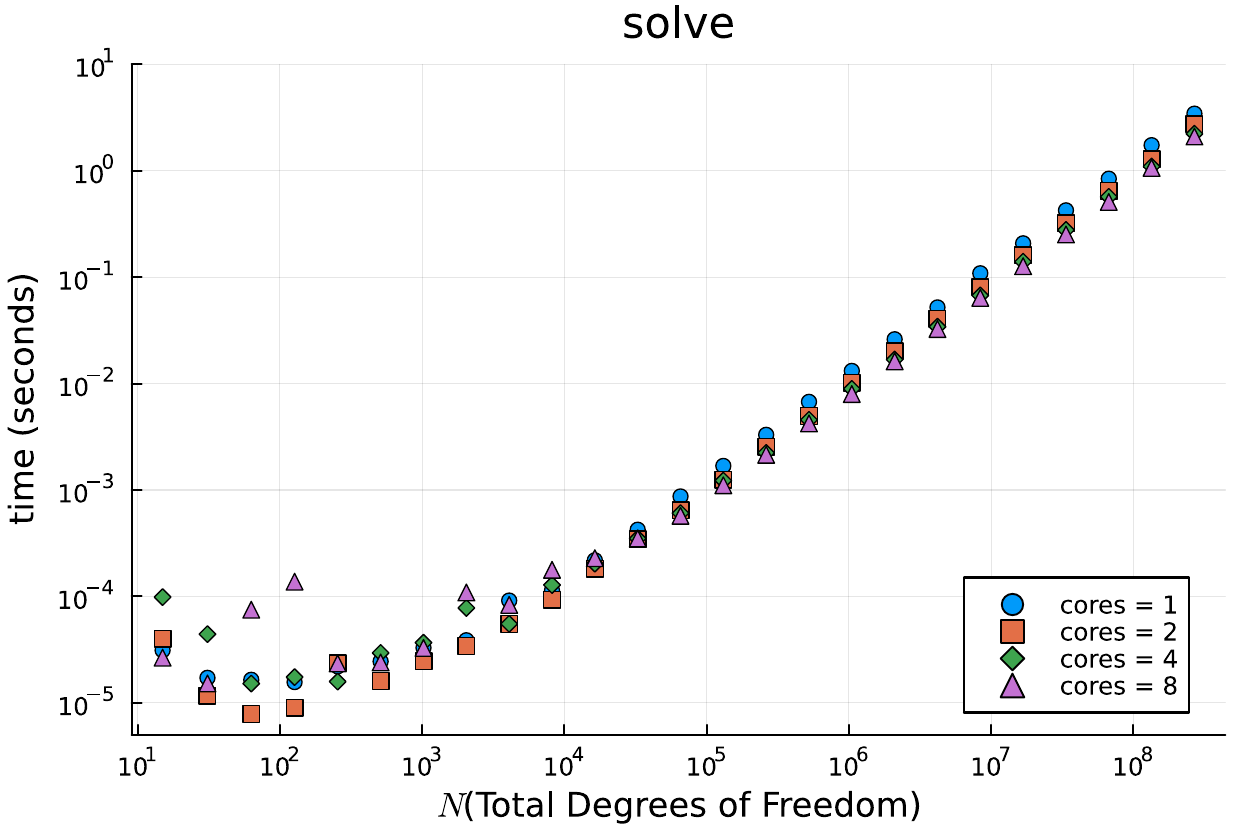}
\caption{Time taken to factorise and solve a 1D (screened) Poisson equation up to degree $p$, where $n= 8$ is the number of elements, where we parallelise over multiple cores, on an M2 MacBook Air with 4 high performance and 4 high efficiency cores, using shared memory. We see substantial speedup for factorisation when utilising the 4 high performance cores whilst a marginal speedup when also using the high efficiency cores.
\label{Figure:onedtims_multithread}}
\end{figure}

\section{A generalised Alternating Direction Implicit (ADI) method}\label{Section:ADI}

Recall from \cref{sec:helmholtz-2d}, the generalised Sylvester equation for the two-dimensional screened Poisson equation (dropping the superscripts $^{\bf x}$ and $^{\bf y}$) is
\begin{align}
\begin{split}
\left(-\Delta_{Y,p} + \frac{\omega^2}{2} M_{Y,p}  \right) U_{pp} M_{Y,p} &+ M_{Y,p} U_{pp}  \left(-\Delta_{Y,p} + \frac{\omega^2}{2} M_{Y,p}  \right)  \\
&= R_{Y,p}^\top M_{P,p} F_{pp} M_{P,p} R_{Y,p} \eqqcolon G_{pp}.
\end{split}
\label{eq:adi:poisson}
\end{align}
Here ${\bf Y}_{0:p}(x) = {\bf C}_{0:p}(x)$ if we consider the screened Poisson equation with a zero Neumann boundary condition and $\omega^2 > 0$. Otherwise ${\bf Y}_{0:p}(x) = {\bf Q}_{0:p}(x)$ if we impose a zero Dirichlet boundary condition and $\omega^2 \geq 0$.

We will solve \eqref{eq:adi:poisson} using a variant of the Alternating Direction Implicit (ADI) method  cf.~\cite{peaceman1955,fortunato2020fast,douglas1963,douglas1962,douglas1964,bruno2010}. The first use of ADI to solve PDEs is attributed to \cite{peaceman1955} and has seen many extensions e.g.~to hyperbolic and nonlinear problems as well as non-Cartesian domains \cite{bruno2010}.  ADI is an iterative approach to approximate $X$ that solves the Sylvester equation $AX - XB = F$, but in a manner that permits  precise error control: given two assumptions on real-valued matrices $A$ and $B$, one is able to explicitly find the number of iterations required for the algorithm to compute $X$ up to a maximum tolerance $\epsilon$. The two assumptions are \cite{fortunato2020fast}:
\begin{enumerate}[label={P}\arabic*.]
\itemsep=0pt
\item $A$ and $B$ are symmetric matrices; \label{as:P1}
\item There exist real disjoint nonempty intervals $[a, b]$ and $[c, d]$ such that $\sigma(A) \subset [a, b]$ and $\sigma(B) \subset [c, d]$, where $\sigma$ denotes the
spectrum of a matrix. \label{as:P2}
\end{enumerate}

\begin{algorithm}[tb]
	\caption{Generalised Alternating Direction Implicit (ADI)}\label{Algorithm:ADI}
	
	{\noindent\bf Input:} {Symmetric matrices $A \in \bbR^{M \times M}$, $B \in \bbR^{N \times N}$,  $C \in \bbR^{N \times N}$,  and  $D \in \bbR^{M \times M}$,  matrix $F \in \bbR^{M \times N}$, tolerance $\epsilon$.}
	
	{\noindent\bf Output:} {Matrix $U \in \bbR^{M \times N}$ satisfying   $A U C -  D U B  \approx F$. }
	
	\smallskip
	{\bf Precomputation}:
\begin{algorithmic}[1]
	\STATE Use the banded symmetric generalised eigenvalue algorithms \cite{crawford1973reduction} to compute generalised eigenvalues  $\sigma(A, D)$  and  $\sigma(B, C)$, where $\sigma(A,B) := \{\lambda : \|(A - \lambda B)^{-1}\| = \infty\}$. The largest and smallest eigenvalues
	give us $a, b, c, d$ such that $\sigma(\tilde A) \subset [a, b]$ and $\sigma(\tilde B) \subset [c, d]$. 
	\STATE Let the number of iterations equal $J = \lceil \log(16 \gamma) \log(4/\epsilon)/\pi^2 \rceil$ where $\gamma = |c-a||d-b|/(|c-b||d-a|)$.
	\STATE Compute the ADI shifts $p_j$ and $q_j$ which have explicit formul\ae\ depending on $\gamma$ \cite[Eq.~(2.4)]{fortunato2020fast}. Notably, we have that $p_j > 0$ and $q_j < 0$ for all $j \in \{1,\ldots,J\}$.
	\FOR{$j = 1,\ldots,J$}
	\STATE Compute the reverse Cholesky factorisations of $ A -  q_j  D $ and $ B - p_j  C $.
	\ENDFOR
	\end{algorithmic}

{\bf Solve}:
\begin{algorithmic}[1]
	\STATE Let $W_0 = \bfzero$
	\FOR{$j = 1,\ldots,J$}
	\STATE Use the precomputed Cholesky factorisations to compute
\begin{align*}
	W_{j-1/2} &=  (F - (A - p_j D)W_{j-1}) (B-p_j C)^{-1},  \\
	W_{j} &=  (A-q_j D)^{-1}(F - W_{j-1/2}(B-q_j C)).  
\end{align*}
	\ENDFOR
	\RETURN $W_J  C^{-1}$.
\end{algorithmic}
\end{algorithm}

The algorithm proceeds iteratively. First one fixes the initial matrix $X_0 = 0$. Then, iteratively for $j =  1,\dots,J$, we compute
\begin{alignat}{2}
\text{for $X_{j-1/2}$ solve \;\;\;\;}& X_{j-1/2} (B - p_j I) &&= F - (A-p_jI) X_{j-1} \label{eq:adi:1},\\ 
\text{for $X_{j}$ solve \;\;\;\;}& (A-q_jI)X_{j} &&= F - X_{j-1/2}(B-q_j I). \label{eq:adi:2}
\end{alignat}

\subsection{Generalised ADI}

In the case of the 2D (Screened) Poisson equation \eqref{eq:weak-form} we have a generalised Sylvester equation which we write in general form as:
\[
	A U C - D U B = F.
\]
We first extend the ADI method to generalised problems in \algref{ADI}. 

\begin{lemma}
Write $C = L^\top L$ and $D = V^\top V$ where $L$ and $V$ are lower triangular.  \algref{ADI} computes $U_J$ satisfying 
\[
	\norm{V (U - U_J) L^\top} \leq \epsilon \norm{V U L^\top}.
\]

\end{lemma}
\begin{proof}
We reduce a generalised Sylvester equation to a standard Sylvester equation as follows:  define
$
X \coloneqq V U L^\top
$
so that our equation becomes
\[
	\underbrace{V^{-\top} A V^{-1}}_{\tilde A} X   -    X \underbrace{L^{-\top} B L^{-1}}_{\tilde B} = \underbrace{V^{-\top}  F L^{-1}}_G.
\]
$\tilde A$ and $\tilde B$ are symmetric matrices whose eigenvalues satisfy $\sigma(\tilde A) =  \sigma(A, D)$  and $\sigma(\tilde B) =  \sigma(B, C)$. The ADI iterations satisfy, for $X_0 = 0$,
\begin{align*}
X_{j-1/2} (\tilde B - p_j I) &= G - (\tilde A-p_jI) X_{j-1}, \\
(\tilde A-q_jI)X_{j} &= G - X_{j-1/2}(\tilde B-q_j I) ,
\end{align*}
where by convergence of the ADI algorithm \cite[Th.~2.1]{fortunato2020fast}:
\begin{align}
\| X - X_J \| \leq \epsilon \| X \|.
\label{eq:adi:8}
\end{align}
Writing $W_j \coloneqq V^{-1} X_j L$ and $W_{j+1/2} \coloneqq V^\top X_{j+1/2} L^{-\top}$ this iteration becomes equivalent to that of \algref{ADI}.  We thus have, for $U_J = W_J C^{-1} = V^{-1} X_J L^{-\top}$, that 
\[
	\norm{V (U - U_J) L^\top}  = \norm{X - X_J}  \leq \epsilon \norm{X} = \epsilon \norm{V U L^\top}.
\]

\end{proof}

In \figref{discont} we show the solution for a discontinuous-right hand side using ADI
with a fixed $h$ and high $p$ to compute the solution. \figref{disconttims} we show the
computational cost of \algref{ADI} for different $h$ and $p$. This shows that in practice
we achieve quasi-optimal complexity, both for the precomputation and the solve. Finally in \figref{neumann} we show the solve time remains quasi-optimal for a zero Neumann boundary condition, and that
the computational cost  improves as $\omega$ increases.

\begin{figure}[ht!]
	\centering
	\includegraphics[width =0.45 \textwidth]{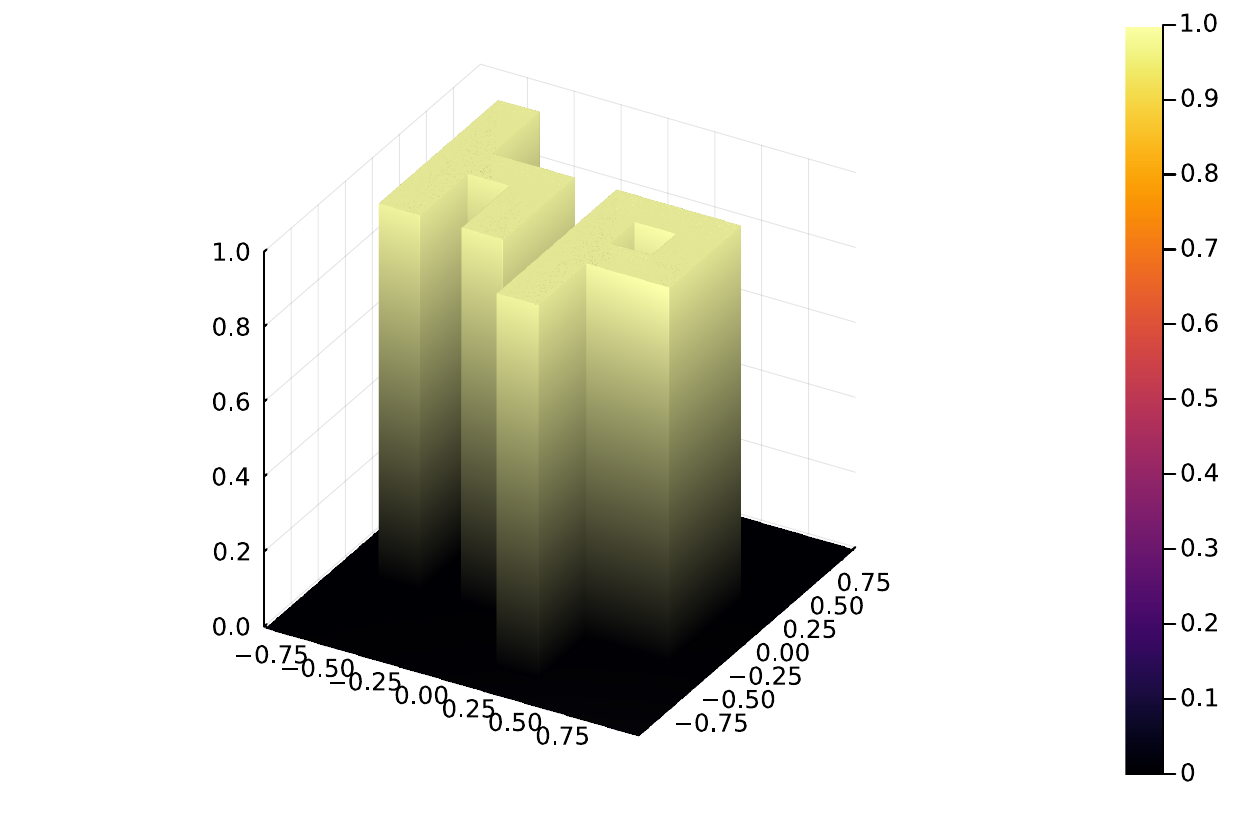}
	\includegraphics[width =0.45 \textwidth]{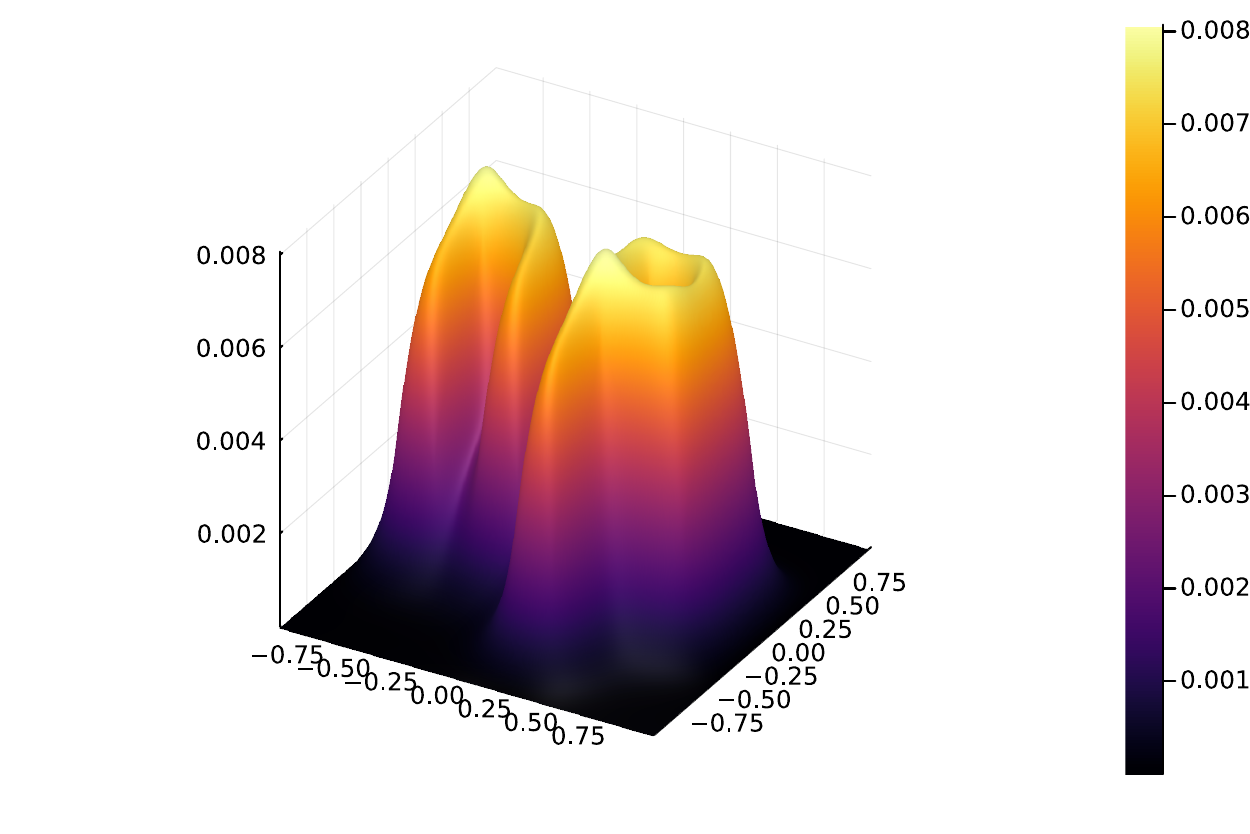}
	\caption{Example PDE with a discontinuous right-hand side $f$ (left) and solution (right) of the screened Poisson equation $-\Delta u+10^2 u  = f$ with a zero Dirichlet boundary condition. By using a $9 \times 9$ elements we can resolve the right-hand side
	exactly, and then use high $p$ to achieve convergence. \label{Figure:discont}}
	\end{figure}

\begin{figure}[ht!]
\centering
\includegraphics[width =0.45 \textwidth]{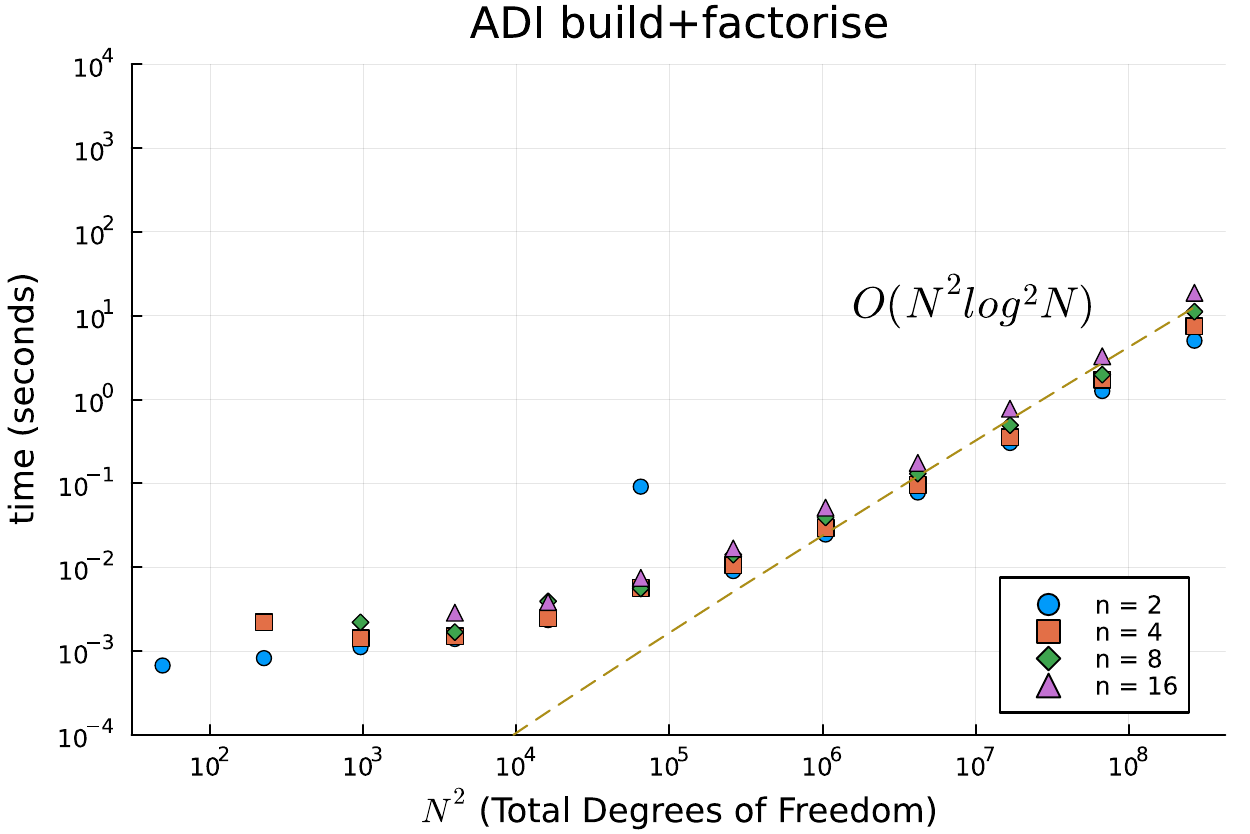}
\includegraphics[width =0.45 \textwidth]{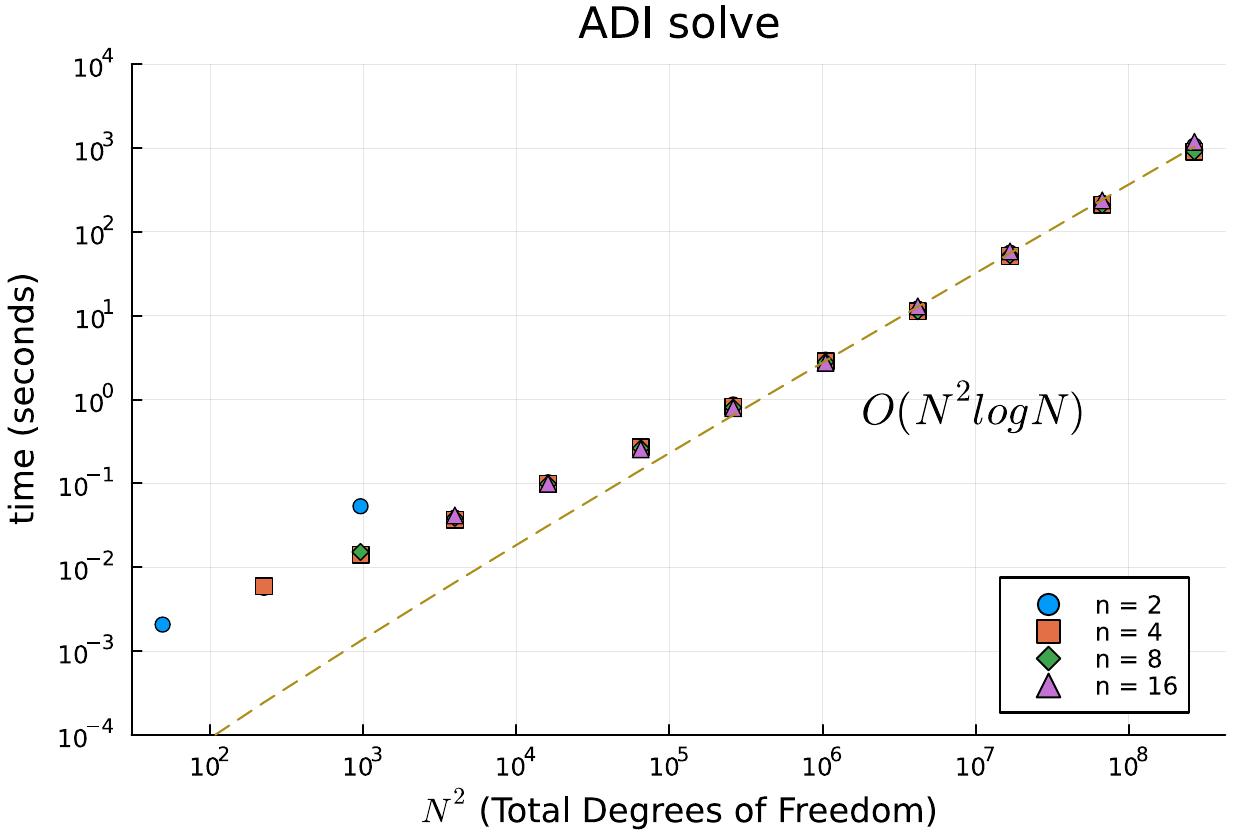}
\caption{Time taken to build/factorise and solve a discretisation of the Poisson equation in 2D using ADI up to degree $p$, where $n$ is the number of elements. The $x$-axis in the second row of timings is the total  number of Degrees of  Freedom (DOF) and demonstrates that the complexity is optimal as either $n (= 2/h)$ or $p$ become large, and largely only depends on the total number of DOF. \label{Figure:disconttims}}
\end{figure}

\begin{figure}[ht!]
\centering
\includegraphics[width =0.45 \textwidth]{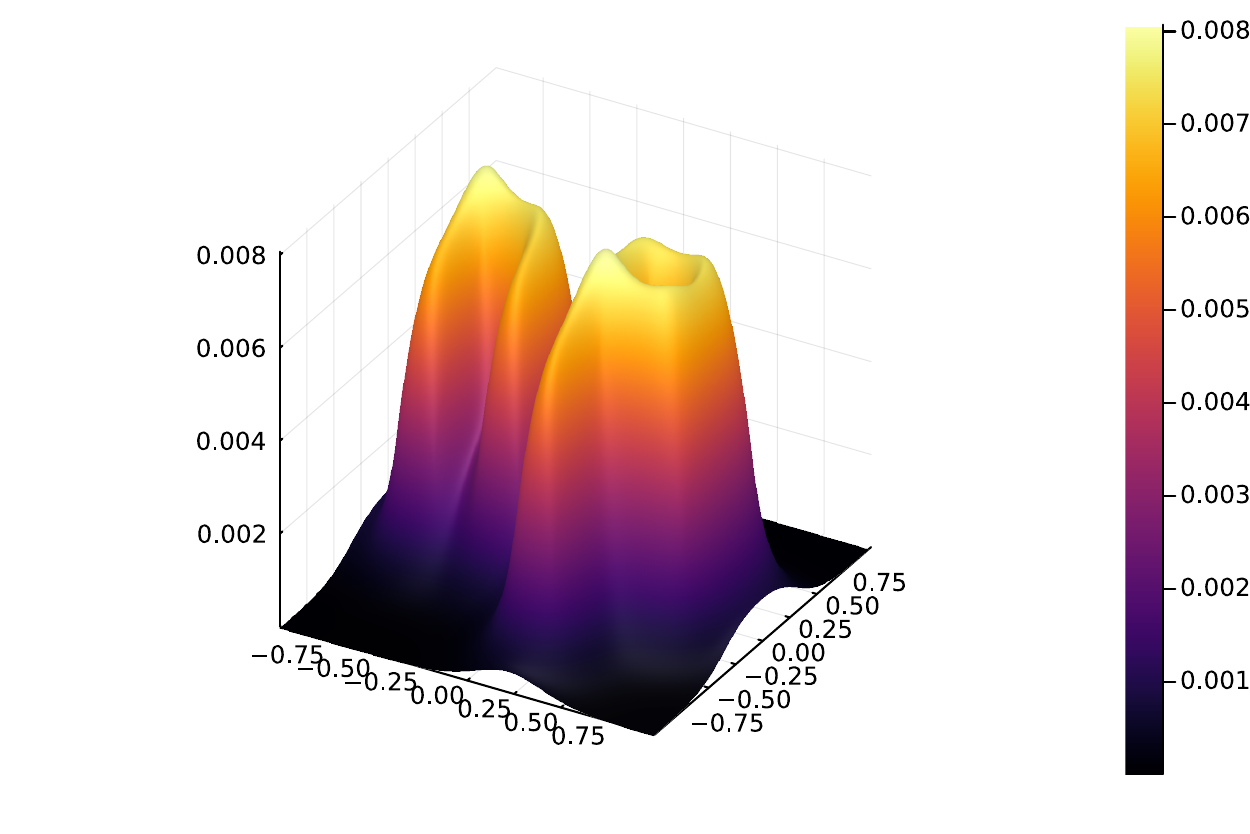}
\includegraphics[width =0.45\textwidth]{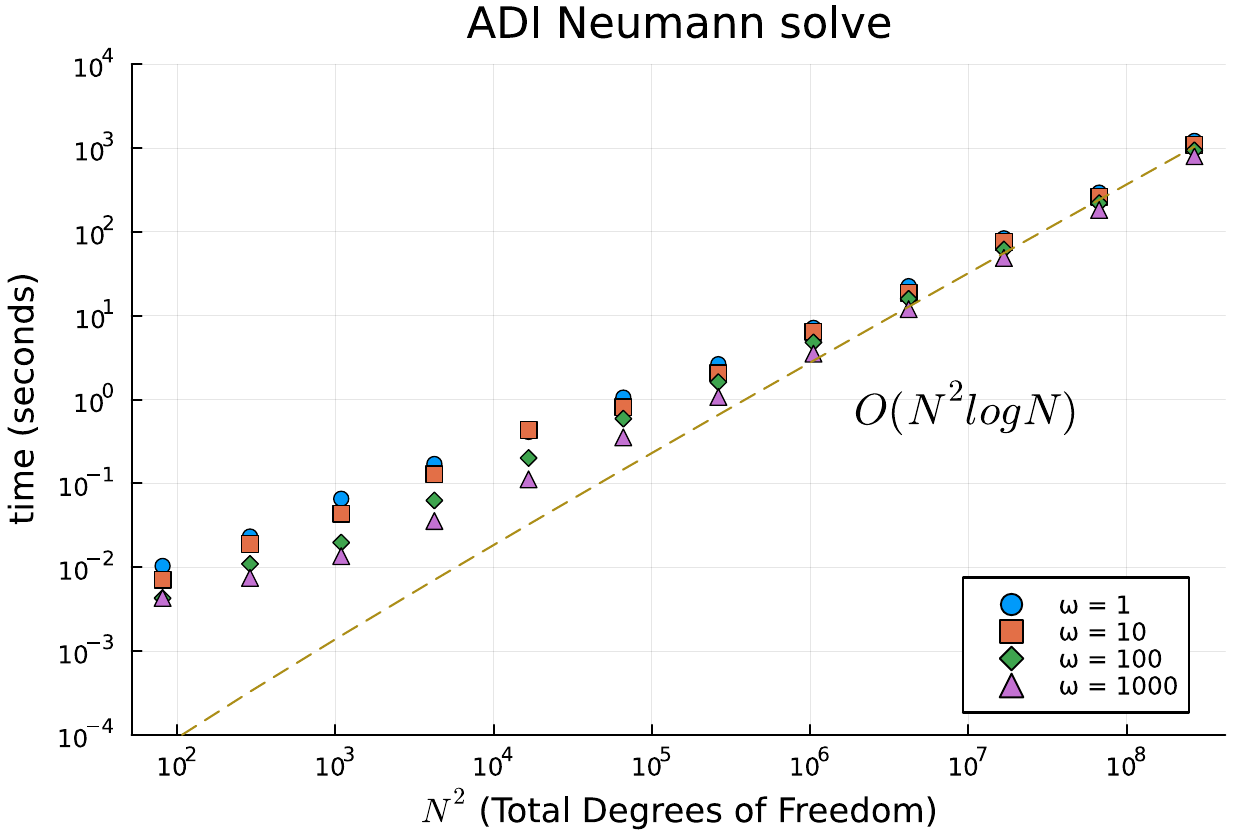}
\caption{Left: Solution of $-\Delta u+10^2 u  = f$ with a zero Neumann boundary condition, using the same right-hand side as in \figref{discont}.
Right: Solve timings for $n = 2$ (two elements) and increasing $p$ with varying choices of $\omega$.\label{Figure:neumann}}
\end{figure}

\subsection{Complexity analysis}

In the previous experiments we observed that \algref{ADI} appears to achieve quasi-optimal complexity. In this section we prove this is guaranteed to be the case. In order to control the complexity, it is necessary to control the number of iterations $J$ which depends on the spectral information of the operators. 
For simplicity, throughout this section we set $p=q$, i.e~we consider an equal discretisation degree in $x$ and $y$. However, we note that all the results generalise. 

A key result to derive the complexity of applying the ADI algorithm solely in terms of $N$ and $p$  are bounds on the spectrum for $L^{-\top} M_{Y,p} L^{-1}$ where $L^\top L = -\Delta_{Y,p} + (\omega^2/2) M_{Y,p}$. This allows us to derive the asymptotic behaviour for $J$.
\begin{lemma}[Spectrum]
\label{Lemma:eigs}
Consider the interval domain $(a,b)$ and a family of quasi-uniform subdivisions $\{\mathcal{T}_h\}_h$ of the interval, where $h$ denotes the mesh size (the minimum diameter of all the cells in the mesh) \cite[Def.~4.4.13]{Brenner2008}. For the (screened) Poisson equation with Neumann boundary conditions consider the quasimatrix ${\bf Y}_{0:p}(x) = {\bf C}_{0:p}(x)$ and with a zero Dirichlet boundary condition consider ${\bf Y}_{0:p}(x) = {\bf Q}_{0:p}(x)$ where $p$ is the truncation degree on each element. Suppose that $L^\top L = A_{Y,p} \coloneqq (-\Delta)_{Y,p} + (\omega^2/2) M_{Y,p}$ where $\omega \neq 0$ in the case of Neumann boundary conditions. Then
\begin{align}
\sigma(L^{-\top} M_{Y,p} L^{-1}) \subseteq \left[\frac{2 h^2}{24 p^{4} + \omega^2 h^2}, C \right],
\end{align}
where $C = \min( (b-a)^2/\pi^2, \max(1, 2/\omega^2))$ in the case of a zero Dirichlet boundary condition and $C = \max(1, 2/\omega^2)$ in the case of a zero Neumann boundary condition.
\end{lemma}
\begin{proof}
The eigenvalue problem to consider is
\begin{align}
L^{-\top} M_{Y,p} L^{-1} {\bf v}_p = \lambda {\bf v}_p,
\label{eq:spec1}
\end{align}
where $\lambda$ and ${\bf v}_p$ denote an eigenvalue and corresponding eigenvector, respectively. First note that $L^{-\top} M_{Y,p} L^{-1}$ is congruent to a symmetric positive-definite matrix and, therefore, $\lambda$ must be real and positive. Left multiplying \cref{eq:spec1} by $L^\top$, and considering ${\bf w}_p = L^{-1} {\bf v}_p$, we deduce that
\begin{align}
M_{Y,p} {\bf w}_p = - \lambda \Delta_{Y,p} {\bf w}_p + \frac{\lambda \omega^2}{2} M_{Y,p}{\bf w}_p.
\label{eq:spec2}
\end{align}
For $w \in H^1(a,b)$, let $|w|_{H^1(a,b)} \coloneqq \|w'\|_{L^2(a,b)}$.  Left multiplying \cref{eq:spec2} by ${\bf w}^\top_p$ implies that
\begin{align}
\| {\bf Y}_{0:p} {\bf w}_p\|^2_{L^2(a,b)} = \lambda | {\bf Y}_{0:p} {\bf w}_p|^2_{H^1(a,b)} + \frac{\lambda \omega^2}{2} \| {\bf Y}_{0:p} {\bf w}_p\|^2_{L^2(a,b)}.
\label{eq:spec3}
\end{align}

\noindent \textbf{(Upper bound).} The upper bound can be split into three cases: (I)  a zero Dirichlet boundary condition, (II) $0<\omega^2 < 2$, and (III) $\omega^2 \geq 2$. In case (I) then ${\bf Y}_{0:p}(a){\bf w}_p = {\bf Y}_{0:p}(b) {\bf w}_p = {\bf Q}_{0:p}(a){\bf w}_p = {\bf Q}_{0:p}(b) {\bf w}_p= 0$. Hence the Poincar\'e inequality (with the optimal Poincar\'e constant) implies that \cite{payne1960optimal}
\begin{align}
\frac{\pi^2}{(b-a)^2}\| {\bf Y}_{0:p} {\bf w}_p\|^2_{L^2(a,b)}  \leq  | {\bf Y}_{0:p} {\bf w}_p|^2_{H^1(a,b)}  \leq \lambda^{-1} \| {\bf Y}_{0:p} {\bf w}_p\|^2_{L^2(a,b)}.
\end{align}
Thus $\lambda \leq (b-a)^2/\pi^2$. In cases (II) and (III) we see that
\begin{align}
\begin{split}
C^{-1} \| {\bf Y}_{0:p} {\bf w}_p\|^2_{L^2(a,b)}
& \leq | {\bf Y}_{0:p} {\bf w}_p|^2_{H^1(a,b)} + \frac{\omega^2}{2} \| {\bf Y}_{0:p} {\bf w}_p\|^2_{L^2(a,b)}\\
&=  \lambda^{-1} \| {\bf Y}_{0:p} {\bf w}_p\|^2_{L^2(a,b)},
\end{split}
\end{align}
where $C^{-1} = \omega^2/2$ in case (II) and $C^{-1} = 1$ in case (III). Thus $\lambda \leq \max(1,2/\omega^2)$. Combining the results from cases (I)--(III), we conclude the upper bound on the spectrum. 

\noindent \textbf{(Lower bound).} Consider a degree $p$ polynomial $\pi_p$  defined on the interval $(0,h)$. Then the following inverse inequality holds \cite[Th.~3.91]{Schwab1998}:
\begin{align}
| \pi_p |_{H^1(0,h)} \leq 2 \sqrt{3} h^{-1} p^2 \|\pi_p\|_{L^2(0,h)}.
\label{eq:spec4}
\end{align}
Consequently, \cref{eq:spec4} implies that
\begin{align}
\begin{split}
\lambda^{-1} \| {\bf Y}_{0:p} {\bf w}_p\|^2_{L^2(a,b)} 
&= | {\bf Y}_{0:p} {\bf w}_p|^2_{H^1(a,b)} + \frac{\omega^2}{2} \| {\bf Y}_{0:p} {\bf w}_p\|^2_{L^2(a,b)} \\
& \leq 12 h^{-2} p^4 \| {\bf Y}_{0:p} {\bf w}_p\|^2_{L^2(a,b)} + \frac{\omega^2}{2} \| {\bf Y}_{0:p} {\bf w}_p\|^2_{L^2(a,b)}.
\end{split}
\label{eq:spec5}
\end{align}
Hence $\lambda  \geq \frac{2 h^2}{24  p^{4} + \omega^2 h^2}$. 
\end{proof}

From the formula of $J$ we arrive at the following:

\begin{lemma}
	Under the conditions of the previous proposition, $ J  = O(\log N \log \epsilon^{-1})$.
\end{lemma}

This allows us to establish complexity results:

\begin{theorem}
	{\bf Precomputation}
	in \algref{ADI} can be accomplished in $O(n N^2 + J N)$ operations,
where we assume we can compute special functions (hyperbolic trigonometric functions, ${\rm log}$, and the elliptic integral ${\rm dn}$) in $O(1)$ operations, where $N$ is the maximal degrees of freedom of either coordinate direction. {\bf Solve} in \algref{ADI}  can be accomplished in $O(J N^2)$ operations. Using the bound on $J$ in the previous result shows quasi-optimal complexity for the precomputation as $p \rightarrow \infty$.
\end{theorem}
\begin{proof}
({\bf Precomputation}): The $B^3$-Arrowhead matrices involved can be viewed as square banded matrices with bandwidth $O(n)$ and dimensions that scale like $O(N)$, hence line (1) can be computed in $O(n N^2)$ operations following \cite{crawford1973reduction}. By the complexity of computing reverse Cholesky factorisations of $B^3$-Arrowhead matrices we know lines (4--6) take $O(J N)$ operations.

({\bf Solve}): Multiplying and inverting $B^3$-Arrowhead matrices can be done on each column of $W_j$ in $O(N)$ operations which immediately gives the result.
\end{proof}

\begin{remark}
Using inverse iteration it is likely that the precomputation cost can be reduced to
$O(N)$ operations but this would require more information on the gap between the eigenvalues. Note also that eigenvalue algorithms have errors which can alter the number of iterations $J$ but we have neglected taking this into consideration as it is unlikely to have a material impact. 
\end{remark}

\section{Transforms and time-evolution}\label{Section:Transforms}

To utilise ADI solvers in an iterative framework for nonlinear elliptic PDEs or in time-evolution problems it is essential to be able to efficiently transform between  values on a grid and coefficients. To accomplish this we need the following transforms in 1D and 2D:

\begin{enumerate}
\item Given a grid, find  the expansion coefficients of the right-hand side into piecewise Legendre polynomials.
\item Given coefficients of the solution in the basis  $\bfQ$, find the values on a grid. 
\end{enumerate} 

The first stage  can be tackled by transforming from values at piecewise Chebyshev grids to
Chebyshev coefficients using the DCT, and thence to Legendre coefficients via a fast Chebyshev--Legendre
transform \cite{alpert1991fast,keiner2011fast,townsend2018fast}. Denote the $p$ Chebyshev points of the first kind as
\[
	\bfx_p^{\rm T} \coloneqq \br[\sin\!\pr(\pi {p-2k+1 \over 2p})]_{k=1}^p.
\]
We denote the transform from Chebyshev points to Legendre coefficients (which combines the DCT with the Chebyshev--Legendre transform) as ${\cal F}_p$ and its inverse as ${\cal F}_p^{-1}$. That is: if
$
f(x) = \bfP_{0:p} \bfc
$
then 
$\bfc = {\cal F}_p f(\bfx_p^{\rm T}).$ Now for multiple elements we affine transform the grid to get a matrix of values. That is, for a matrix of grid points $X_p^n = [\bfx_p^1 | \cdots | \bfx_p^n]$ we transform each column:
$
{\cal F}_p f(X_p^n)
$. Reinterpreting this matrix as a block-vector, whose rows correspond to blocks, gives the coefficients in the basis $\bfP^\bfx$. That is, we use
\[
	{\rm vec}(({\cal F}_p f(X_p^n))^\top)
\]
where
$
{\rm vec} : \bbR^{p \times n} \rightarrow \bbR^{p n}
$
is the operator from matrices to (block) vectors that concatenates the columns.

To extend this to two dimensions, we use the grids $\bfx = {\rm vec}(X_p^n)$ and $\bfy = {\rm vec}(X_q^m)$ and hence we want to transform from a matrix of values on the tensor product grid i.e., 
\[
	F \coloneqq \br[f(x_k,y_j)]_{k=1,j=1}^{k=pn,j=qm}.
\]
The 2D  transform is then
$
{\cal F}_p^n F ({\cal F}_q^m)^\top.
$

The second stage can be accomplished by first computing the coefficients in a piecewise Legendre basis via applying the matrix $R^\bfx$,
transforming to Chebyshev coefficients via  a fast Legendre--Chebyshev transform, then applying the inverse DCT to recover the values on a Chebyshev grid. That is, if we have
\[
	u(x,y) = \bfQuptop(x) U \bfQuptoq(y)^\top
\]
then we can transform back to a grid via
\[
	u(\bfx, \bfy^\top) = ({\cal F}_p^n)^{-1}  R^\bfx U (R^\bfy)^\top {\cal F}_q^{-\top}.
\]

\begin{figure}[ht!]
	\centering
	\includegraphics[width =0.45 \textwidth]{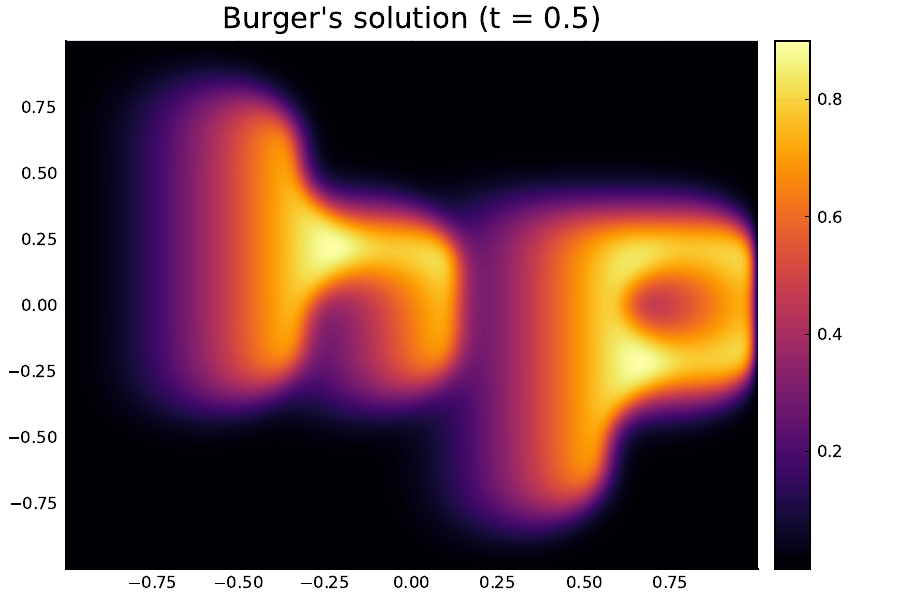}
	\includegraphics[width =0.45 \textwidth]{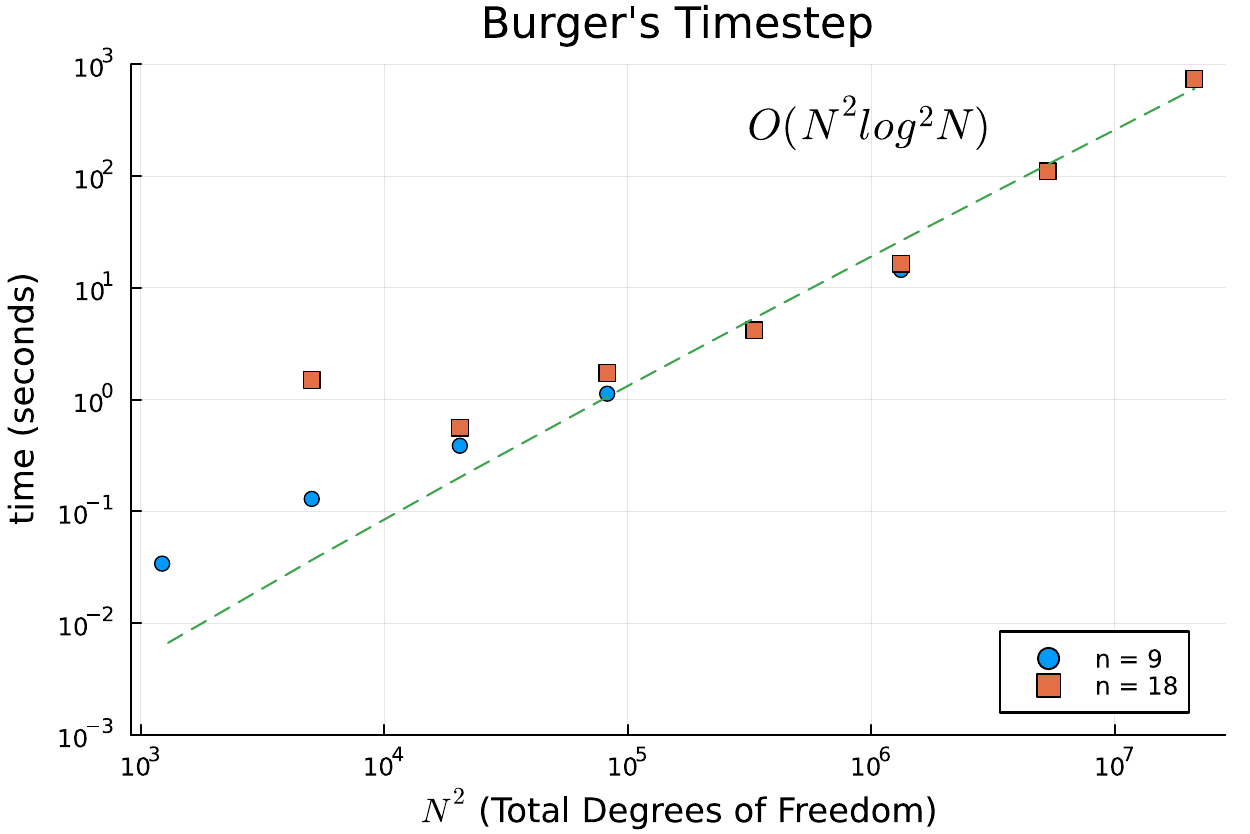}
	\caption{Left: Solution  of Burgers' equation $u_t + u u_x = \epsilon \Delta u$ with a zero Dirichlet boundary condition, where the initial condition is the discontinuous right-hand side as in \figref{discont},
	with $\epsilon = 0.1$. 
	Right: time taken for a single time-step, using an implicit-explicit splitting method where the linear part is solved using implicit Euler and the nonlinear part with explicit Euler, after transforming back to a grid.
	\label{Figure:burgers}}
\end{figure}

As an example of the utility of fast transforms,  \figref{burgers}  considers the classic Burgers' equation:
\[
	(u_t+ u u_x)(x,y,t) = \epsilon \Delta u(x,y,t)
\]
with a zero Dirichlet boundary condition and a discontinuous initial condition. We discretise in time using a simple implicit-explicit splitting method, taking a linear step via implicit Euler followed by a nonlinear step via explicit Euler:
\begin{align*}
u_{k+1/2}(x,y) &= (I - (\delta t)  \epsilon \Delta)^{-1} u_k(x,y),\\
u_{k+1}(x,y) &= (u_{k+1/2} + (\delta t) u_{k+1/2,x} u_{k+1/2})(x,y).
\end{align*}
We represent $u_k(x,y)$ as a matrix $U_k$ containing coefficients in an expansion of  tensor products of piecewise Legendre polynomials, i.e., using the basis $\bfP^\bfx$. The half time-steps $u_{k+1/2}(x,y)$ are then represented as a matrix $U_{k+1/2}$ giving coefficients in tensor products of $\bfQ^\bfx$, where the coefficients are computed using ADI as described above.  To determine $u_{k+1/2}(x,y)$ on a grid we simply convert down to Legendre and then apply the inverse fast Legendre transform, that is  values are approximated by:
\[
	V_{k+1/2} \coloneqq {\cal F}^{-1}R U_{k+1/2}R^\top {\cal F}^{-\top}.
\]
For $u_{(k+1/2),x}(x,y)$ we compute its Legendre coefficients using  the derivative matrix
alongside the conversion matrix:
\[
	V_{k+1/2,x} \coloneqq {\cal F}^{-1}D U_{k+1/2}R^\top {\cal F}^{-\top}.
\]
We can then determine the Legendre coefficients $U_{k+1}$ as
\[
	U_{k+1} \coloneqq {\cal F} \br[V_{k+1/2} + (\delta t) V_{k+1/2,x} \otimes V_{k+1/2}].
\]
The right-hand side plot in \cref{Figure:burgers} roughly demonstrates the predicted $O(N^2 \log^2 N)$ complexity.

\section{ADI as a preconditioner}
\label{Section:Preconditioning}

In this section, we explore how our ADI-based solver may be used as a preconditioner for an iterative Krylov method to tackle problems with variable coefficients, including an example with a singularity. In particular one may use a graded  mesh to isolate the singularity so that the variable coefficient is well-approximated by piecewise polynomials. Experimentally we see that the ADI preconditioner is $hp$-robust and we, therefore, retain the quasi-optimal complexity of the solver as $h \to 0$ and $p \to \infty$. 

\begin{remark}[Non-Cartesian cells and curved boundaries]
The preconditioning strategy outlined in this section may also be used for discretisations of elliptic problems with non-Cartesian cells or posed on domains with curved boundaries. The underlying idea utilizes equivalent operator preconditioning \cite{axelsson2009}. We omit a description or implementation in this work but refer the interested reader to an excellent introduction by Brubeck and Farrell in \cite[Sec.~2.7]{Brubeck2022} who prove the preconditioner is spectrally equivalent to the original problem. See also \cite{couzy1995,fischer2000,witte2021}.
\end{remark}

Consider the domain $\Omega = (-1,1)^2$ and the following singular variable coefficient PDE problem:
\begin{align}
(-\Delta  - 10 \log \sqrt{x^2+y^2}) u(x,y) = 1 \;\; \text{with} \; u|_{\partial \Omega} = 0.
\label{eq:variable-coeff}
\end{align}
Note that the variable coefficient $- 10 \log \sqrt{x^2+y^2} \to \infty$ as $|(x,y)|_{\ell^2} \to 0$, i.e.~at the origin in the centre of the domain. After an FEM discretisation, the residual may be evaluated in a matrix-free manner via a quasi-optimal complexity transform as outlined in \secref{Transforms}. In particular, let $G \coloneqq [- 10 \log \sqrt{x_k^2+y_j^2}]_{k=1,j=1}^{k=pn,j=pn}$. Then given the FEM coefficient matrix $U$ of a 2D FEM function $u_{p}$, we may approximately evaluate $F \coloneqq \langle v_p, (- \log \sqrt{x^2+y^2}) u_{p} \rangle$, for all basis functions $v_{p}$, via
\begin{align}
F = \mathrm{vec}(R^{-1} M_P \mathcal{F}[ G \odot (\mathcal{F}^{-1} R U R^\top \mathcal{F}^{-\top})] \mathcal{F}^\top M_P R^{-\top}),
\label{eq:evaluation}
\end{align}
where $\odot$ denotes the Hadamard product (element-wise multiplication between two matrices). We emphasize \cref{eq:evaluation} is computed in quasi-optimal complexity. This motivates the use of an iterative Krylov method to solve the $hp$-FEM discretisation of \cref{eq:variable-coeff}. We opt for the conjugate gradient method (CG) preconditioned with the inverse discretised weak Laplacian matrix applied via the quasi-optimal ADI strategy of \secref{ADI} with the tolerance $10^{-4}$.

We use a tensor-product mesh graded towards the origin with the cells endpoints
\begin{align}
\mathcal{T}_m = (-1, -10^{-1}, \dots, 10^{-m}, 0, 10^{-m}, \dots, 10^{-1}, 1)^2
\label{eq:graded-mesh}
\end{align}
for a given $m \geq 1$. We plot the solution in \cref{Figure:Preconditioner} and provide the number of preconditioned CG iterations for various $p$ and number of cells in \cref{Tab:Preconditioner}. We observe $hp$-robustness--the number of preconditioned CG iterations is independent of the degree and number of cells of the mesh--leading to a quasi-optimal complexity solver for \cref{eq:variable-coeff}

\begin{figure}[ht!]
\centering
\includegraphics[width =0.4 \textwidth]{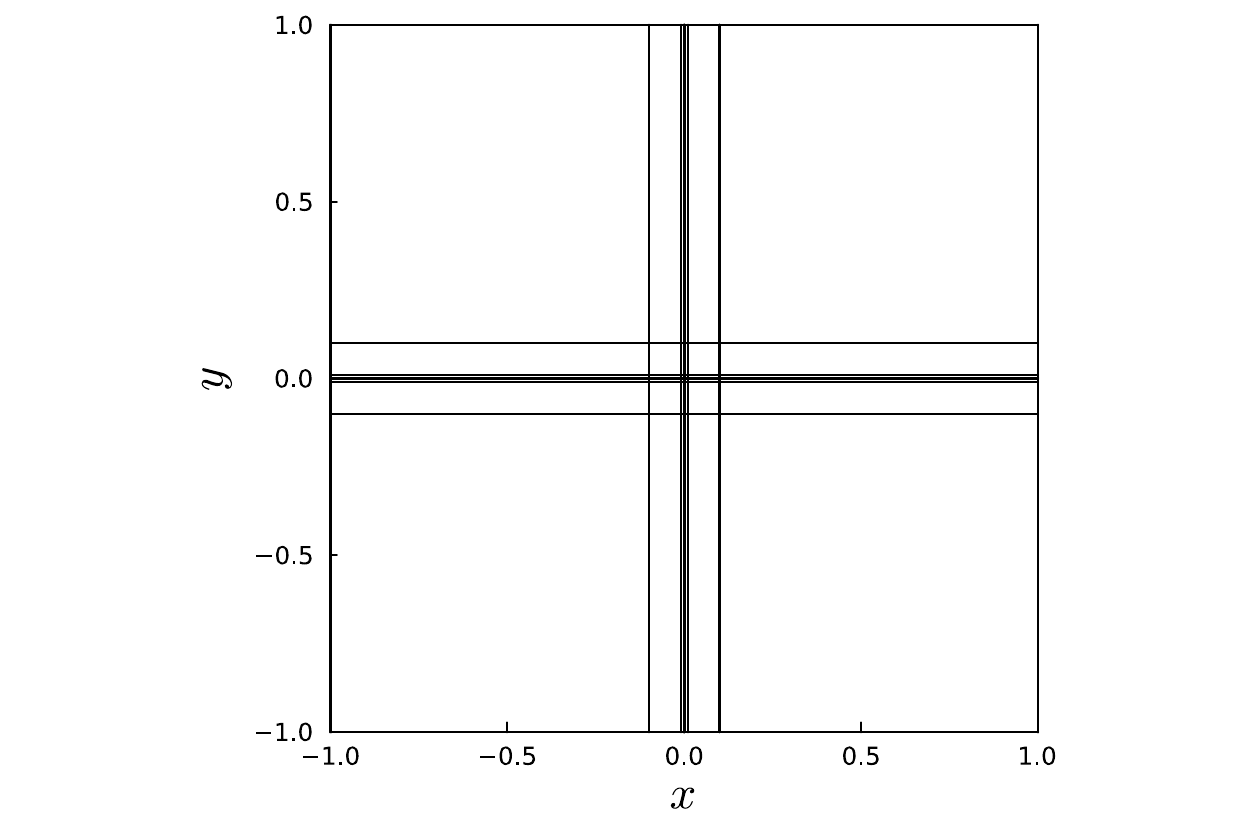}
\includegraphics[width =0.48 \textwidth]{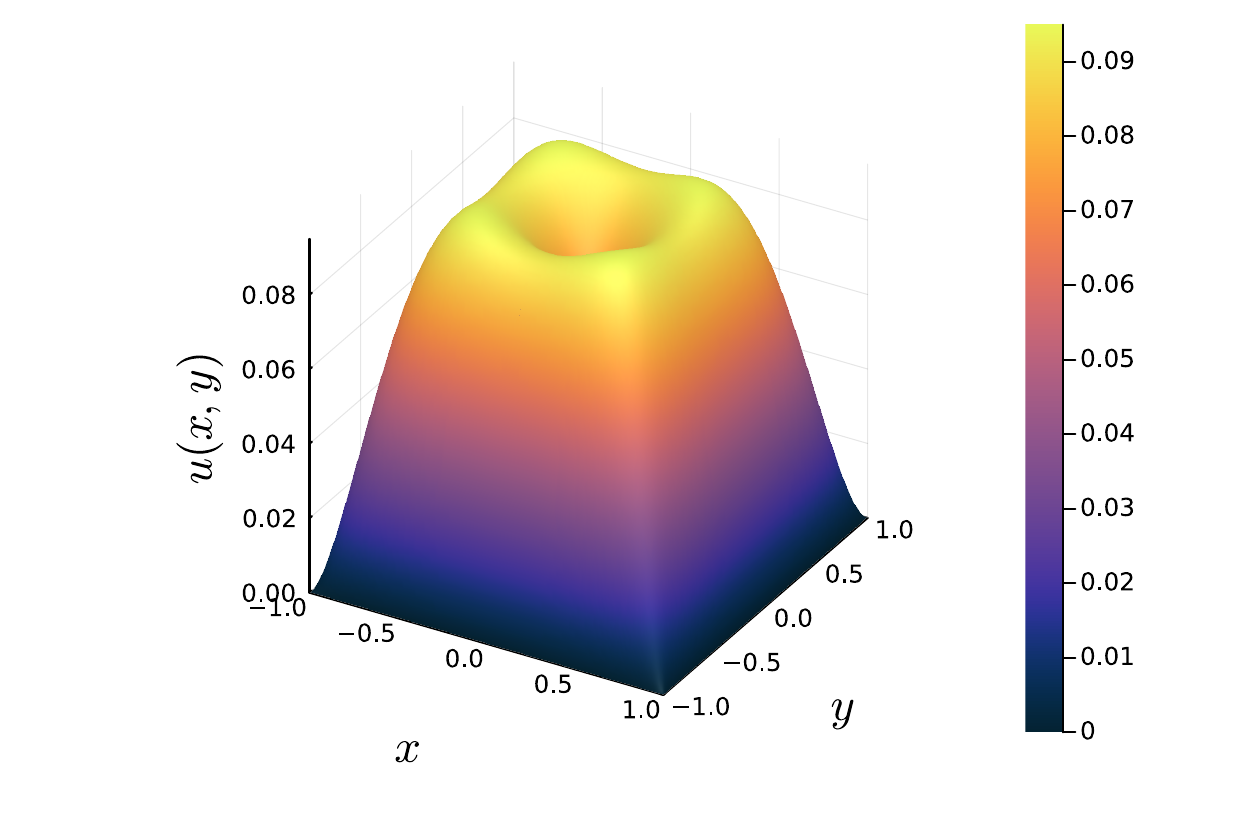}
\caption{ Left: Graded tensor-product mesh \cref{eq:graded-mesh} with $m=3$. Right: Solution of singular variable coefficient problem \cref{eq:variable-coeff} with partial (pre-tensor) degree $p=128$ and graded mesh \cref{eq:graded-mesh} with $m=3$.}
\label{Figure:Preconditioner}
\end{figure}

\begin{table}[ht]
\small
\centering
\begin{tabular}{|c|ccccc|}
\hhline{~|-----|}
\multicolumn{1}{c|}{} & \multicolumn{5}{c|}{$p$}\\
\hline
\rule{0pt}{2.5ex} $m$ (\# cells) & $2^3$ &$2^4$ & $2^5$ & $2^6$ & $2^7$  \\
\hline
1 (16)&8 &7&7&7&7\\
2 (36)& 7&7&7&7&7\\
3 (64)& 7&7&7&7&7\\
\hline
\end{tabular}%
\caption{ The number of preconditioned CG iterations to solve the $hp$-FEM discretisation of \cref{eq:variable-coeff} to a relative tolerance of $10^{-8}$ with increasing partial degree $p$ and number of cells (in brackets) in the graded mesh \cref{eq:graded-mesh} as controlled by the parameter $m$. The preconditioner is the inverse weak Laplacian matrix applied via the quasi-optimal ADI strategy of \secref{ADI} with tolerance $10^{-4}$.} 
\label{Tab:Preconditioner}
\end{table}

\section{Future work}

We have constructed the first provably quasi-optimal complexity $hp$-FEM method for the (screened) Poisson equation on a rectangle, built on taking advantage of the sparsity structure. There are some clear extensions to this work:

\begin{enumerate}
\item For non-positive definite but symmetric operators, it is possible to do $L^\top DL$ factorisations of the $B^3$-Arrowhead matrices in optimal complexity. However, this may lead to ill-conditioning. Unfortunately, stable factorisations such as $QL$  only achieve $O(p n + n^3)$ complexity as there is fill-in in the top blocks. 
\item  Extensions to quasi-optimal solves on cylinders and box domains is possible. Cylinders have already been considered in \cite{papadopoulos2024} and a box domain could be tackled via a nested ADI approach as discussed in \cite[Sec.~5]{fortunato2020fast}. However, Fortunato and Townsend note that the convergence of the nested approach is highly sensitive to the termination tolerances of the inner solves. Subsequently a more robust approach via tensor-trains was introduced in \cite{shi2021}. Our techniques directly translate to the tensor-train context, however, the implementation is nontrivial and left for future work. 
\item  A fully parallelised implementation with distributed memory in 2D and 3D. To minimise communication between workers this would involve distributing on elements, but in principle the ADI steps can be largely parallelised with only minimal communication corresponding to data on the interface needed to be shared. 
\item  Extensions to higher-order time-steppers.  For linear problems,  Backward Differentiation Formul\ae\ (BDF)  are a natural choice for taking advantage of the efficient solves of our spacial discretisation whilst achieving higher order accuracy. For nonlinear time-evolution problems, a straightforward extension would be to use Strang splitting, which would achieve the same computational complexity as the simple implicit-explicit time-stepper considered  with higher order accuracy.     
\end{enumerate}

\section*{Acknowledgments}
We would like to thank Dan Fortunato, Marcus Webb, and Matt Colbrook. IP would like to thank Pablo Brubeck for their discussion on optimal complexity $p$-multigrid methods. SO and IP were  supported by an EPSRC grant (EP/T022132/1). IP was also funded by the Deutsche Forschungsgemeinschaft (DFG, German Research
Foundation) under Germany's Excellence Strategy -- The Berlin Mathematics
Research Center MATH+ (EXC-2046/1, project ID: 390685689).

\bibliographystyle{siamplain}
\bibliography{optimalcomplexityhp}

\appendix

\section{Recurrences}\label{Appendix:Recurrences}

In this appendix we provide the formul\ae\ for the entries in the matrices \cref{rec:lowering}, \cref{rec:conversion1}--\cref{rec:conversion2} and \cref{rec:deriv1}--\cref{rec:deriv2}.
From  \cite[18.9.8]{DLMF}, we have that
\begin{align}
W_k(x) = {1 \over 2k+3}(P_k(x) - P_{k+2}(x)). \label{app:1}
\end{align}
Thus we deduce that the entries in \cref{rec:lowering} are
\[
\underbrace{[W_0,W_1,W_2,\ensuremath{\ldots}]}_\bfW = \underbrace{[P_0,P_1,P_2,\ensuremath{\ldots}]}_\bfP \underbrace{\begin{bmatrix} 1/3 \\ 0 & 1/5 \\ -1/3 & 0 & 1/7 \\ &-1/5 &0 & 1/9 \\ &&\ensuremath{\ddots} & \ensuremath{\ddots} & \ensuremath{\ddots} 
\end{bmatrix}}_{L_W}.
\]

Next we derive the entries in \cref{rec:conversion1}--\cref{rec:conversion2}. Consider the reference cell $(-1,1)$. Then there exists two hat functions with nonzero support, $h_0(x) = (1-x)/2$ and $h_1(x) = (x+1)/2$. Since these are degree one polynomials then, for $k \geq 2$, $\langle P_k, h_j \rangle = 0$, and hence $R_{k0} = {\bf 0}$. Moreover, we have that $M_{00} =2$, $M_{11} = 2/3$, $\langle P_0, h_j \rangle = 1$, and $\langle P_1, h_j \rangle = (-1)^{j+1}/3$ for $j \in \{0,1\}$. A scaling argument reveals that these entries are independent of the size of the element. Hence $R_{k0} \in \mathbb{R}^{n \times(n+1)}$ and the entries in \cref{rec:conversion1} are
\begin{align}
R_{00} = 
\begin{bmatrix}
1/2 & 1/2 & &\\
& \ddots & \ddots & \\
& & 1/2 & 1/2
\end{bmatrix},
\;\;
R_{10} = 
\begin{bmatrix}
-1/2 & 1/2 & &\\
& \ddots & \ddots & \\
& & -1/2 & 1/2
\end{bmatrix},
\end{align}
Moreover, for $j > 0$, $R_{kj} \in \mathbb{R}^{n \times n}$ and from \cref{app:1} we deduce the entries in \cref{rec:conversion2} are
\begin{align}
-R_{(k+2)j} = R_{kj} = 
\begin{bmatrix}
\frac{1}{1+2j} & &\\
& \ddots & \\
& &   \frac{1}{1+2j}  
\end{bmatrix} \;\; \text{if} \;\; k = j \pm 1,
\end{align}
and otherwise $R_{kj} = {\bf 0}$.

To compute the entries in \cref{rec:deriv1}--\cref{rec:deriv2}, consider the reference cell $(-1,1)$ and note that $h_0'(x) = -1/2$, $h_1'(x) = 1/2$ and $W_k'(x) = -P_{k+1}(x)$, cf.~\cref{eq:deriv}. Let $\delta_i = x_{i} - x_{i-1}$ for $i \in \{1:n\}$. Then, by a scaling argument, we deduce that
\begin{align}
\hspace{5mm} D_{00} = 
\begin{bmatrix}
-1/\delta_1 & 1/\delta_1 & &&\\
& -1/\delta_2 & 1/\delta_2 && \\
& &\ddots &\ddots & \\
& &&\hspace{-5mm} -1/\delta_n & 1/\delta_n
\end{bmatrix} \in \mathbb{R}^{n \times (n+1)}
\end{align}
and, for $k>0$,
\begin{align}
D_{kk} = 
\begin{bmatrix}
-2/\delta_1 & & &\\
&-2/\delta_2  & &\\
& &\ddots & \\
& & & -2/\delta_n
\end{bmatrix} \in \mathbb{R}^{n \times n} \;\; \text{with $D_{kj} = {\bf 0}$ if $k \neq j$}.
\end{align}

\end{document}

%% file: ex_shared.tex
\usepackage{lipsum}
\usepackage{amsfonts}
\usepackage{graphicx}
\usepackage{epstopdf}
\usepackage{algorithmic}
\ifpdf
  \DeclareGraphicsExtensions{.eps,.pdf,.png,.jpg}
\else
  \DeclareGraphicsExtensions{.eps}
\fi

\newsiamremark{remark}{Remark}
\newsiamremark{hypothesis}{Hypothesis}
\crefname{hypothesis}{Hypothesis}{Hypotheses}
\newsiamthm{claim}{Claim}

\headers{Quasi-optimal Complexity $hp$-FEM}{K. Knook, S. Olver, I. P. A. Papadopoulos}

\title{Quasi-optimal Complexity $hp$-FEM\thanks{Submitted to the editors DATE.
\funding{This work was funded by the Fog Research Institute under contract no.~FRI-454.}}}

\author{Kars Knook\thanks{Department of Mathematics, University of Oxford, England
  (\email{kars.knook@maths.ox.ac.uk}).}
\and Sheehan Olver\thanks{Department of Mathematics, Imperial College, London, England
  (\email{s.olver@imperial.ac.uk}, \url{https://www.ma.imperial.ac.uk/\string~solver/}).}
\and Ioannis P. A. Papadopoulos\thanks{Weierstrass Institute for Applied Analysis and Stochastics, Berlin, Germany (\email{papadopoulos@wias-berlin.de})  }}

\usepackage{amsopn}

%% file: somacros.tex
\def\addtab#1={#1\;&=}

\def\meeq#1{\def\ccr{\\\addtab}
 \begin{align*}
 \addtab#1
 \end{align*}
  }  
  
  \def\leqaddtab#1\leq{#1\;&\leq}

\def\pr(#1){\left({#1}\right)}
\def\br[#1]{\left[{#1}\right]}
\def\fbr[#1]{\!\left[{#1}\right]}

\def\ip<#1>{\left\langle{#1}\right\rangle}
\def\iip<#1>{\left\langle\!\langle{#1}\right\rangle\!\rangle}

\def\norm#1{\left\| #1 \right\|}

\def\fpr(#1){\!\pr({#1})}

\def\mapengine#1,#2.{\mapfunction{#1}\ifx\void#2\else\mapengine #2.\fi }

\def\map[#1]{\mapengine #1,\void.}

\def\mapenginesep_#1#2,#3.{\mapfunction{#2}\ifx\void#3\else#1\mapengine #3.\fi }

\def\mapsep_#1[#2]{\mapenginesep_{#1}#2,\void.}

\def\vcbr{\br}

\def\bvect[#1,#2]{
{
\def\dots{\cdots}
\def\mapfunction##1{\ | \  ##1}
\begin{pmatrix}
		 \,#1\map[#2]\,
\end{pmatrix}
}
}

\def\vect[#1]{
{\def\dots{\ldots}
	\vcbr[{#1}]
}}

\def\vectt[#1]{
{\def\dots{\ldots}
	\vect[{#1}]^{\top}
}}

\def\Vectt[#1]{
{
\def\mapfunction##1{##1 \cr} 
\def\dots{\vdots}
	\begin{pmatrix}
		\map[#1]
	\end{pmatrix}
}}

\def\D{{\rm d}}
\def\dx{\D x}

\def\tF_#1{{\tt F}_{#1}}

\def\tFC_#1{{\tt T}_{#1}}

\def\appref#1{Appendix~\ref{Appendix:#1}}

\def\secref#1{Section~\ref{Section:#1}}

\def\figref#1{Figure~\ref{Figure:#1}}
\def\algref#1{Algorithm~\ref{Algorithm:#1}}

\def\qqand{\qquad\hbox{and}\qquad}
\def\qfor{\quad\hbox{for}\quad}
\def\qqfor{\qquad\hbox{for}\qquad}

\def\elllRpz_#1{\ell_{#1{\rm z}}^{(\lambda,R),p}}

\def\bbR{{\mathbb R}}